\documentclass[11pt,a4paper]{article}

\usepackage[margin=1truein]{geometry}

\usepackage{lmodern}
\usepackage[T1]{fontenc}
\usepackage{textcomp}
\usepackage{bm}

\usepackage[group-separator={,}]{siunitx}
\usepackage{amsmath}
\usepackage{amssymb}
\usepackage{amsthm}
\usepackage{subcaption}
\usepackage{physics}
\usepackage{mathtools}
\DeclarePairedDelimiter{\prn}{(}{)}
\DeclarePairedDelimiter{\set}{\{}{\}}
\DeclarePairedDelimiterX{\Set}[2]{\{}{\}}{\,{#1}\;\delimsize|\;{#2}\,}
\DeclarePairedDelimiter{\inpr}{\langle}{\rangle}
\usepackage{mleftright}\mleftright

\usepackage{nicematrix}

\usepackage{algorithm}
\usepackage[noend]{algpseudocode}
\algnewcommand{\algorithmicinput}{\textbf{Input:}}
\algnewcommand{\Input}{\item[\algorithmicinput]}
\algnewcommand{\algorithmicoutput}{\textbf{Output:}}
\algnewcommand{\Output}{\item[\algorithmicoutput]}
\algnewcommand{\algorithmicbreak}{\textbf{break}}
\algnewcommand{\Break}{\State{\algorithmicbreak}}
\algnewcommand{\nil}{\textbf{nil}}

\usepackage{etoolbox}
\cslet{blx@noerroretextools}\empty%
\usepackage[
    date=year,
    isbn=false,
    natbib,
    url=false,
    sorting=nyt,
    style=trad-abbrv
]{biblatex}

\usepackage{booktabs}
\usepackage{multirow}

\usepackage{enumitem}
\setlist{font=\upshape,leftmargin=*}
\setlist[1]{labelindent=\parindent}
\setlist[enumerate,1]{label={(\arabic*)}}

\usepackage[
    bookmarksnumbered,
    colorlinks,
    hypertexnames=false,
    pdfdisplaydoctitle,
    pdfusetitle,
    unicode
]{hyperref}
\usepackage{url}

\usepackage{upref}
\usepackage[capitalize,noabbrev]{cleveref}

\expandafter\def\csname ver@etex.sty\endcsname{3000/12/31}
\usepackage{autonum}

\usepackage{comment}
\usepackage[color={red!100!green!33},textsize=scriptsize]{todonotes}
\newtheorem{theorem}{Theorem}[section]
\newtheorem{lemma}[theorem]{Lemma}
\newtheorem{proposition}[theorem]{Proposition}

\theoremstyle{definition}

\crefname{step}{Step}{Steps}

\newcommand{\ones}{\mathbf{1}}
\newcommand{\Rp}{\R_{\ge 0}}
\newcommand{\Rpp}{\R_{> 0}}

\newcommand{\R}{\mathbb{R}}

\newcommand{\given}{\mathrel{|}}
\newcommand{\Ord}{\mathrm{O}}
\newcommand{\moffdiag}{M_\mathrm{offdiag}}
\newcommand{\mdiag}{M_\mathrm{diag}}
\newcommand{\compress}{\star}

\newcommand{\textscup}[1]{\textsc{\textup{#1}}}

\DeclareMathOperator*{\argmax}{arg\,max}

\DeclareMathOperator{\diag}{diag}
\DeclareMathOperator{\Diag}{Diag}

\DeclareMathOperator{\nnz}{nnz}
\DeclareMathOperator{\err}{err}

\newcommand{\tmax}{t_\mathrm{max}}

\title{\texorpdfstring{%
  Fast and Numerically Stable Implementation \\
  of Rate Constant Matrix Contraction Method
}{%
  Fast and Numerically Stable Implementation of Rate Constant Matrix Contraction Method
}}

\author{
  Shinichi Hemmi%
  \texorpdfstring{\thanks{
    Preferred Networks, Inc., Tokyo 100-0004, Japan.
    E-mail: \href{mailto:hemmi.shinichi@gmail.com}{\nolinkurl{hemmi.shinichi@gmail.com}}
  }}{}
  \and Satoru Iwata%
  \texorpdfstring{\thanks{
    Department of Mathematical Informatics, Graduate School of Information Science and Technology, University of Tokyo, Tokyo 113-8656, Japan.
    E-mail: \href{maialto:iwata@mist.i.u-tokyo.ac.jp}{\nolinkurl{iwata@mist.i.u-tokyo.ac.jp}}}\;}{}%
  \footnotemark[2]
  \and Taihei Oki%
  \texorpdfstring{%
    \thanks{Institute for Chemical Reaction Design and Discovery (ICReDD), Hokkaido University, Sapporo, Hokkaido 001-0021, Japan. E-mail: \href{mailto:oki@icredd.hokudai.ac.jp}{\nolinkurl{oki@icredd.hokudai.ac.jp}} (T.\ Oki)}%
  }{}
}

\begin{document}
\maketitle

\begin{abstract}

The rate constant matrix contraction (RCMC) method, proposed by Sumiya et al.~(2015, 2017), enables fast and numerically stable simulations of chemical kinetics on large-scale reaction path networks.
Later, Iwata et al.~(2023) mathematically reformulated the RCMC method as a numerical algorithm to solve master equations whose coefficient matrices, known as rate constant matrices, possess the detailed balance property.

This paper aims to accelerate the RCMC method. The bottleneck in the RCMC method lies in the greedy selection of steady states, which is actually equivalent to the greedy algorithm for maximum a posteriori (MAP) inference in determinantal point processes (DPPs) under cardinality constraints.
Hemmi et al.~(2022) introduced a fast implementation of the greedy DPP MAP inference, called LazyFastGreedy, by combining the greedy algorithm of Chen et al.~(2018) with the lazy greedy algorithm by Minoux~(1978), a practically efficient greedy algorithm that exploits the submodularity of the objective function.
However, for instances arising from chemical kinetics, the straightforward application of LazyFastGreedy suffers from catastrophic cancellations due to the wide range of reaction time scales.

To address this numerical instability, we propose a modification to  LazyFastGreedy that avoids the subtraction of like-sign numbers by leveraging the properties of rate constant matrices and the connection of the DPP MAP inference to Cholesky decomposition.
For faster implementation, we utilize a segment tree, a data structure that manages one-dimensional arrays of elements in a semigroup.
We also analyze the increase in relative errors caused by like-sign subtractions and permit such subtractions when they do not lead to catastrophic cancellations, aiming to further accelerate the process.
Using real instances from chemical reactions, we confirm that the proposed algorithm is both numerically stable and significantly faster than the original RCMC method.

\end{abstract}

\section{Introduction}

\subsection{Background}
Simulation of chemical kinetics is a pivotal tool for predicting chemical reactions on reaction path networks.
A reaction path network consists of $n$ equilibrium states (EQs) and paths connecting two EQs through transient states (TSs).
The state of a chemical reaction at time $t \ge 0$ is represented as the \emph{yield vector} $x(t)$, which is an $n$-dimensional vector whose $v$th $(v \in V \coloneqq \set{1, \dotsc, n})$ component $x_v(t)$ corresponds to the yield (amount) of the $v$th EQ.
The dynamics of the yield vector follow the \emph{master equation}
\begin{align}\label{def:master}
  \dot{x}(t) = Kx(t).
\end{align}

In~\eqref{def:master}, $K = {(K_{uv})}_{u,v \in V}$ is an $n \times n$ \emph{rate constant matrix}, which is determined from the potential energies of EQs and TSs if the system is described by the microcanonical ensemble~\cite{Sumiya2020-ti,Sumiya2015-br}.
The rate constant matrix $K$ satisfies the following conditions~\cite{Iwata2023}:
\begin{enumerate}[label={(RCM\arabic*)}]
  \item $K_{uv} \ge 0$ for $u \ne v$,\label{item:rcm1}
  \item $K_{vv} = -\sum_{u \ne v} K_{uv}$ for 
 $v \in V$,\label{item:rcm2}
  \item there exists a positive vector $\pi = {(\pi_v)}_{v \in V}$ such that $K_{uv} \pi_v = K_{vu} \pi_u$ for $u, v \in V$.\label{item:rcm3}
\end{enumerate}
The $(u, v)$ off-diagonal entry $K_{uv}$ of $K$ represents the reaction rate constant from the $v$th EQ to the $u$th EQ, and is non-negative by~\ref{item:rcm1}.
The second condition~\ref{item:rcm2} comes from the conservation law of the sum of yields, i.e., $\sum_{v \in V} x_v(t)$ is constant for all $t$.
The last condition~\ref{item:rcm3} and the vector $\pi$ are referred to as the \emph{detailed balance condition} and a \emph{stationary} (or \emph{equilibrium}) \emph{distribution}, respectively.
The yield vector $x(t)$ converges to $\pi$ as $t \to \infty$ regardless of the initial yield $x(0)$ if the network is connected~\cite{Anderson1991-ec}.
The conditions~\ref{item:rcm1}--\ref{item:rcm3} can be rephrased to state that $L = -K\Pi$ is a (weighted) \emph{graph Laplacian matrix}, where $\Pi$ is the diagonal matrix with diagonal entries $\pi_1, \dotsc, \pi_n$.
In probability theory, this system of $x(t)$ is known as the \emph{continuous-time Markov chain}, and the master equation~\eqref{def:master} is called the \emph{Kolmogorov forward equation}.

Rate constant matrices are typically ill-conditioned due to the wide range of reaction time scales.
Furthermore, recent advancements in automatic chemical reaction pathway search techniques, such as the anharmonic downward
distortion following (ADDF) method~\cite{Ohno2004-iw} and the single-component artificial force induced reaction (SC-AFIR) method~\cite{Maeda2014-cm}, have resulted in large-scale ($n \gtrsim \num{10000}$), sparse ($\nnz(K) \lesssim \num{30000}$), and extremely ill-conditioned rate constant matrices; the ratio of the largest to smallest non-zero off-diagonal entries of $K$ can exceed $10^{200}$ (see \cref{tbl:dataset}).
Numerically solving such a stiff differential equation has been recognized as a challenging task.
While one can obtain $x(t)$ for every $t$ from the eigendecomposition of $K$ via the explicit formulation of the analytic solution $x(t) = \mathrm{e}^{tK}x(0)$, computing the eigendecomposition of ill-conditioned matrices requires time-consuming multi-precision arithmetic as the double-precision arithmetic suffers from heavy catastrophic cancellations.

To address these issues, Sumiya et al.~\cite{Sumiya2015-br,Sumiya2017-qo} presented the \emph{rate constant matrix contraction} (RCMC) method for approximately solving master equations described by the canonical or microcanonical ensemble.
The RCMC method relies on neither finite difference, multi-precision arithmetic, nor subtraction of like-sign numbers.
The original description~\cite{Sumiya2015-br,Sumiya2017-qo} of the RCMC method was based on chemical intuitions.
Subsequently, \citet{Iwata2023} provided a mathematical reformulation of the RCMC method, showing that it works for the master equation~\eqref{def:master} with any coefficient matrix $K$ satisfying~\ref{item:rcm1}--\ref{item:rcm3} and performing a theoretical error analysis.
Recent work~\citep{Harabuchi2024-mh} applies the reformulated RCMC method to the sensitivity analysis of complicated reaction path networks.

The reformulated procedure of the RCMC method is outlined as follows~\citep[Section~3]{Iwata2023}.
The input of the RCMC method is a rate constant matrix $K$, an initial yield vector $p$, and the end time $\tmax$ of the simulation.
The RCMC method runs in two steps: Steps~1 and~2.
Step~1 constructs a sequence of growing sets $\emptyset = S^{(0)} \subsetneq S^{(1)} \subsetneq \dotsb \subsetneq S^{(k)} \subsetneq V$ of EQs with $0 \le k \le n$ in a greedy fashion.
This procedure is essentially the pivoted Gaussian elimination for Cholesky decomposition on $L = -K\Pi$ with a minor modification to avoid subtraction of like-sign numbers.
Intuitively, every EQ $v \in S^{(j)}$ is regarded as ``steady'' after some time $t^{(j)}$, i.e., $\dot{x}_v(t) \approx 0$ for $t \ge t^{(j)}$.
On the other hand, EQs in $T^{(j)} \coloneqq V \setminus S^{(j)}$ are said to be ``transient''.
The value $k$ is the number of EQs that are regarded as steady at $\tmax$.
In typical use such as $\tmax = \SI{86400}{sec}$ ($=\text{one day}$), $k$ roughly satisfies $0.4n \lesssim k \lesssim 0.7n$.
The idea of bipartitioning EQs by steadiness comes from the classical \emph{quasi-steady state approximation} (QSSA).

Step~2 of the RCMC method computes a vector $q^{(j)}$ and the time $t^{(j)}$ for $j = 1, \dotsc, k$ by explicit matrix calculations.
It is guaranteed that $0 < t^{(1)} < \dotsb < t^{(k)} \le \tmax$ and each $q^{(j)}$ is a non-negative vector with $\sum_{v \in V} q^{(j)}_v = \sum_{v \in V} p_v$.
The value $t^{(j)}$ is a reference time when the EQs in $S^{(j)}$ and $T^{(j)} \coloneqq V \setminus S^{(j)}$ are regarded as steady and transient, respectively, and the vector $q^{(j)}$ serves as an approximation to $x(t^{(j)})$.
This is a ``full'' scheme to obtain a rough sketch of the entire trajectory of the yield vector $x(t)$ through $t \in [0, \tmax]$.
If only the yield vector $x(\tmax)$ at the last moment $\tmax$ is needed, it suffices to compute only the last vector $q^{(k)}$.
We call the former and the latter way the \texttt{full} and \texttt{last} options, respectively.
See \cref{sec:rcmc} for more on the RCMC method.

The advantages of the RCMC method lie in its numerical stability and running time.
Regarding numerical stability, the RCMC method is free from subtractions of like-sign numbers, and thus, catastrophic cancellation never occurs.
In fact, it is reported in~\cite[Section~6]{Iwata2023} that using double-precision floating-point numbers is enough for obtaining a good approximation to $x(t)$.
The paper~\cite{Iwata2023} also observed that the RCMC method with \texttt{full} option runs within seconds for rate constant matrices with $n \approx \num{1700}$.\footnote{%
  Several versions of the RCMC methods are presented in~\citep{Iwata2023}.
  In this paper, we employ the standard Type~A method with the ``diag'' option, in which the RCMC method coincides with the original description in~\citep {Sumiya2015-br,Sumiya2017-qo}.
}

\subsection{Aim of This Paper}
The aim of this paper is to speed up the (reformulated) RCMC method.
This is motivated by the fact that large-scale rate constant matrices with $n \gtrsim 10,000$ are recently obtained by GRRM and AFIR.
For such large-scale matrices, the RCMC method took more than 5 minutes, even with the faster \texttt{last} option (see \cref{sec:experiments}).
Since the RCMC method is repeatedly applied in the kinetics-based navigation in AFIR~\cite{Sumiya2020-ti}, its speeding up is crucial for accelerating the entire procedure of AFIR.

Step~1 of the RCMC method runs in $\Theta(kn^2)$ time.
Step~2 takes $\Theta(k^3)$ and $\Theta(k^2)$ times for \texttt{full} and \texttt{last} options, respectively, assuming that $k = \tilde{\Omega}(\sqrt{n})$ and $K$ is a sparse matrix, i.e., the number of non-zero entries in $K$ is $\tilde{\Theta}(n)$ (see \cref{sec:runtime-rcmc}).
Here, $\tilde{\Omega}$ and $\tilde{\Theta}$ hides logarithmic factors.
This means that Step~1 is the bottleneck of the RCMC method.
We have also confirmed in numerical experiments that Step~1 actually takes a large portion of the running time for small $k$; see \cref{sec:experiments}.

Our starting point for accelerating Step~1 lies in its connection to the greedy algorithm for \emph{maximum a posteriori} (MAP) \emph{inference} for \emph{determinantal point processes} (DPPs) under cardinality constraints.
A DPP is a probability distribution on the power set $2^V$ of a finite set $V = \set{1, \dotsc, n}$ determined from an $n \times n$ positive semi-definite matrix $L$~\citep{Macchi_1975}.
The probability of observing $X \subseteq V$ is proportional to $\det L_{XX}$, where $L_{XX}$ denotes the principal submatrix of $L$ with rows and columns indexed by $X$.
The MAP inference for DPPs under the cardinality constraint, which is to compute $X$ with maximum $\det L_{XX}$ subject to $|X| = k$, has many applications in machine learning, such as recommendation systems~\cite{Chen2018-nk}, document summarization~\cite{Kulesza2012-er}, and diverse molecular selection~\cite{Nakamura2022-nd}.
While exact solving the MAP inference is NP-hard~\citep{Ko_1995}, a greedy algorithm provides a good approximation due to the submodularity of the function $X \mapsto \log \det L_{XX}$; see \cref{sec:dpp-map-inference} for detail.

We then observe in \cref{sec:relation} that Step~1 of the RCMC method is essentially equivalent to a slight generalization of the greedy MAP inference for DPPs.
Therefore, any implementation of the greedy MAP inference for DPPs can be used as that of Step~1.
We refer to the original implementation~\citep{Iwata2023} of Step~1 as \textsc{Greedy}.
\citet{Chen2018-nk} presented an algorithm, called \textsc{FastGreedy} in~\cite{Hemmi2022-zj}, for faster implementation of greedy MAP inference, improving the running time from $\Theta(kn^2)$ to $\Theta(k^2 n)$.
\textsc{FastGreedy} is based on Doolittle’s algorithm~\cite{doolittle1878method} for Cholesky decomposition, which computes entries of the Cholesky factor directly.
Combining \textsc{FastGreedy} with the lazy greedy algorithm~\cite{Minoux_1978} for cardinality-constrained submodular function maximization, \citet{Hemmi2022-zj} proposed a practically faster implementation, called \textsc{LazyFastGreedy}.
Intuitively, \textsc{LazyFastGreedy} reduces the running time by delaying the computation of entries in the Cholesky factor that are not immediately needed.
The running time of \textsc{LazyFastGreedy} depends on how well this lazy heuristic works.
In the best case, \textsc{LazyFastGreedy} runs in $\Theta(k^3)$ time and is not slower than \textsc{FastGreedy} even in the worst case (up to logarithmic terms).
In our numerical experiment, \textsc{FastGreedy} ran faster than \textsc{Greedy}, and \textsc{LazyFastGreedy} outperformed \textsc{FastGreedy} in real instances.
See \cref{sec:fast-greedy} for details on \textsc{FastGreedy} and \textsc{LazyFastGreedy}.

Unfortunately, however, both \textsc{FastGreedy} and \textsc{LazyFastGreedy} suffer from numerical instability.
In fact, our numerical experiment reported in \cref{sec:experiments} indicates that the set $S^{(j)}$ computed by these two algorithms with double-precision floating-point numbers was incorrect at some $j$ because of the accumulated cancellation errors.

\subsection{Our Contributions}
To achieve fast and numerically stable computation, we present a modification of \textsc{LazyFastGreedy}, termed \textsc{StableLazyFastGreedy}.
As stated above, Step~1 of the RCMC method is essentially the Cholesky decomposition of the graph Laplacian matrix $L = -K\Pi$.
The only point where \textsc{LazyFastGreedy} may subtract like-sign numbers is in the computation of diagonal entries of the Cholesky factor $C$ of $L$.
Due to the conditions~\ref{item:rcm1}--\ref{item:rcm3} on $K$, the Cholesky factor $C$ has the property that its off-diagonal entries are non-positive, and every diagonal entry is the minus of the sum of the off-diagonal entries with the same column.
Therefore, we can retrieve each $v$th diagonal entry of $C$ by summing all the non-zero off-diagonal entries in the same column, i.e., $C_{vv} = -\sum_{v \in T'} C_{uv}$ with $T' \coloneqq T^{(j)} \setminus \set{v}$.
As $C_{uv} \le 0$, this way of computing $C_{vv}$ does not involve additions of opposite-sign numbers.
This is, however, incompatible with the approach of \textsc{LazyFastGreedy}, which benefits from delaying the computation of off-diagonal entries of $C$.

Our idea for overcoming this issue is to consider the \emph{compression} $L^{[j,v]}$ of $L$ by $T'$, a matrix obtained by aggregating the rows (columns) of $L$ indexed by $T'$ into a single row (column).
The compression $L^{[j,v]}$ has the properties that (i) the Cholesky factor $C^{[j,v]}$ of $L^{[j,v]}$ is the compression of $C$ by $T'$ and (ii) $L^{[j,v]}$ is again a graph Laplacian matrix.
The first property implies that we can obtain $C_{vv}$ as the negative of an off-diagonal entry $C^{[j,v]}_{\compress v}$, where $\compress$ denotes the index of the compressed row.
Furthermore, the second property enables us to regard the compression $L^{[j,v]}$ as an input matrix instead of $L$ and obtain $C^{[j,v]}_{\compress v}$ in the same way as other off-diagonal entries of the Cholesky factor.
If the lazy heuristics work well, this reduces the computational cost compared to computing $C_{uv}$ for all $u \in T'$.

In implementing \textsc{StableLazyFastGreedy}, the values of compressed entries $L^{[j,v]}_{\compress w} = \sum_{u \in T'} L_{uw}$ for all $w \in S^{(j)} \cup \set{v}$ are required for computing  $C^{[j,v]}_{\compress v}$.
To calculate the compressed entries only with additions of non-negative numbers, we employ a data structure called a \emph{segment tree} (or a \emph{range tree})~\cite{Chazelle1988-zw}.
A segment tree manages a one-dimensional array of $n$ entries of a semigroup and allows $\Ord(\log n)$-time queries for executing the semigroup operation on a consecutive subarray and updating an array element.
Using segment trees regarding the set of non-negative reals with the addition as a semigroup, we can compute the compressed entries for successive $j = 0, 1, 2, \dotsc$ efficiently.

If the lazy heuristics works the best, \textsc{StableLazyFastGreedy} runs in $\Theta(k^3)$ time (ignoring logarithmic factors), which is as asymptotically fast as \textsc{LazyFastGreedy}.
However, if the lazy heuristics works poorly, the worst-case time complexity swells up to $\Theta(k^3 n)$.
To reduce this complexity in practice, we introduce \textsc{RelaxedStableLazyFastGreedy}, which implements additional heuristics called the \emph{relaxing heuristics}.
Relaxing heuristics aims to accelerate the algorithm by allowing the subtraction of like-sign numbers if their ratio is smaller than a prescribed threshold.
To bound errors that may occur by subtractions, for two positive numbers $\hat{a}$ and $\hat{b}$ containing errors from their true values $a$ and $b$, respectively, we give an upper bound on the relative error of $\hat{a} - \hat{b}$ to the true difference $a - b$.
Through our error analysis, we can determine the threshold depending on the error tolerance.
When the relaxing heuristics works well, i.e., much computation has been done with subtractions, \textsc{RelaxedStableLazyFastGreedy} runs in $\Theta(k^2 n)$ time even if the lazy heuristics work the worst.
We summarize the methods, their running times, and numerical stability in \cref{tbl:running-time}.

\begin{table}[t]
    \centering
    \begin{tabular}{cccccc}\toprule
       Method        & Lazy & Relax & Running time & Stability \\\midrule
       \textsc{Greedy}~\cite{Iwata2023} & ---  & --- & $\Theta(kn^2)$ & \checkmark \\
       \textsc{FastGreedy}~\cite{Chen2018-nk} & --- & --- & $\Theta(k^2 n)$ & \\
       \textsc{LazyFastGreedy}~\cite{Hemmi2022-zj} & best  & --- & $\Theta(k^3)$ & \\
                             & worst & --- & $\Theta(k^2 n)$ & \\
       \textsc{StableLazyFastGreedy} & best  & --- & $\Theta(k^3)$ & \checkmark \\
                                   & worst & --- & $\Theta(k^3 n)$ & \checkmark \\
       \textsc{RelaxedStableLazyFastGreedy} & best  & best & $\Theta(k^3)$ & \checkmark \\
                                   & best  & worst & $\Theta(k^3)$ & \checkmark \\
                                   & worst & best & $\Theta(k^2 n)$ & \checkmark \\
                                   & worst & worst & $\Theta(k^3 n)$ & \checkmark \\
    \bottomrule\end{tabular}
    \caption{
        Comparison of the running time and the numerical stability.
        In running time analyses, $k = \Omega\prn[\big]{{\nnz(K)}^{1/3}}$ is assumed, where $\nnz(K) = \Omega(n)$ denotes the number of non-zero entries in $K$.
        Check marks on the stability column mean that the method does not cause catastrophic cancellations by avoiding subtracting reals close to each other.
    }\label{tbl:running-time}
\end{table}

Finally, we numerically compare the methods with double-precision arithmetic for six rate constant matrices obtained from chemical reaction path networks. 
We observe that \textsc{LazyFastGreedy} and \textsc{FastGreedy} fail before $t^{(j)}$ reaches $\tmax = \SI{86400}{sec}$ due to catastrophic cancellations, whereas the other three methods succeed.
For the largest data with $n = \num{12215}$ and $k = \num{8056}$ corresponding to $\tmax = \SI{86400}{sec}$, \textsc{Greedy} and \textsc{RelaxedStableLazyFastGreedy} (with a reasonable error tolerance) took $\num{1041}$ seconds and $154$ seconds, respectively.
That is, \textsc{RelaxedStableLazyFastGreedy} ran about seven times as fast as \textsc{Greedy}.
We also observe that both the lazy and relaxing heuristics work well in practice by measuring the number of arithmetic operations omitted due to the heuristics.

\subsection{Organization}
The remainder of this paper is organized as follows.
\Cref{sec:preliminaries} describes the RCMC method and its connection to the greedy MAP inference for DPPs.
\Cref{sec:fast-greedy} explains \textsc{FastGreedy} and \textsc{LazyFastGreedy}, providing their running time analyses.
We then present \textsc{StableLazyFastGreedy} in \cref{sec:proposed} and \textsc{RelaxedStableLazyFastGreedy} in \cref{sec:rslfg}.
\Cref{sec:experiments} provides results on numerical experiments.
Finally, \cref{sec:conclusion} concludes this paper.

\section{RCMC Method and Greedy MAP Inference for DPPs}\label{sec:preliminaries}
This section introduces the RCMC method in \cref{sec:rcmc} and discusses its running time and numerical stability in \cref{sec:runtime-rcmc}.
Then, \cref{sec:dpp-map-inference} describes the greedy MAP inference for DPPs, and \cref{sec:relation} explains its relationship to the RCMC method.

\paragraph{Notations.}
Let $\R$, $\Rp$, and $\Rpp$ denote the set of reals, non-negative reals, and positive reals, respectively.
For an non-negative integer $n$, we denote $[n] \coloneqq \set{1, \dotsc, n}$ and $[0, n] \coloneqq \set{0, 1, \dotsc, n}$.
Note that $[0] \coloneqq \varnothing$.

For a finite set $U$, let $\R^U$ denote the set of all $|U|$-dimensional real vectors whose components are indexed by $U$.
Similarly, for finite sets $U$ and $V$, we denote by $\R^{U \times V}$ the set of real matrices of size $|U| \times |V|$ with rows and columns indexed by $U$ and $V$, respectively.
We simply denote $\R^n \coloneqq \R^{[n]}$, $\R^{U \times m} \coloneqq \R^{U \times [m]}$, and $\R^{n \times m} \coloneqq \R^{[n] \times [m]}$.
For a matrix $A \in \R^{U \times V}$, let $A_{I, J}$ and $A_{IJ}$ denote the submatrix of $A$ whose rows and columns are indexed by $I \subseteq U$ and $J \subseteq V$, respectively.
We simply denote $A_{\set{i}J}$, $A_{I\set{j}}$, and $A_{\set{i}\set{j}}$ as $A_{iJ}$, $A_{Ij}$, and $A_{ij}$, respectively.
For a vector $a \in \R^U$, we denote by $a_I$ its subvector indexed by $I \subseteq U$.
In particular, if $I = \set{i}$, we denote it by $a_i$.

For two (row or column) vectors $a, b \in \R^U$, we mean their (standard) inner product by $\inpr{a, b} \coloneqq \sum_{i \in U} a_i b_i$.
We denote the 2-norm of $a \in \R^U$ by $\norm{a} \coloneqq \sqrt{\inpr{a, a}}$.
Let $\ones$ denote the all-one vector of an appropriate dimension.
For a square matrix $L \in \R^{V \times V}$, let $\diag(L)$ denote the vector in $\R^V$ obtained by extracting diagonal entries of $L$.
Conversely, for a vector $d \in \R^V$, let $\Diag(d)$ denote the $V \times V$ diagonal matrix obtained by arranging the components of $d$ into the diagonals.

The number of non-zero entries of a matrix $A$ is denoted by $\nnz(A)$.
We say that a matrix $A \in \R^{U \times V}$ is \emph{dense} if $\nnz(A_{IJ}) = \Theta(|I||J|)$ and \emph{sparse} if $\nnz(A_{IJ}) = \tilde{\Theta}(|I| + |J|)$ for any $I \subseteq U$ and $J \subseteq V$, where $\tilde{\Theta}(\cdot)$ hides logarithmic factors.

\subsection{RCMC Method}\label{sec:rcmc}
Throughout this paper, $V$ is a finite set of cardinality $n$.
A \emph{rate constant matrix} over $V$ is an $n \times n$ matrix $K \in \R^{V \times V}$ satisfying \ref{item:rcm1}--\ref{item:rcm3}.
The vector $\pi = \prn{\pi_v}_{v \in V}$ in~\ref{item:rcm3} is called a \emph{stationary distribution}.
If we set $\Pi \coloneqq \Diag(\pi)$, we can express a rate constant matrix as $K = -L\Pi^{-1}$, where $L \in \R^{V \times V}$ is a symmetric matrix with $L_{uv} \le 0$ for $u,v\in V$ where $u \ne v$ and $L_{vv} = -\sum_{u \ne v} L_{uv}$ for $v \in V$.
Such a matrix $L$ is positive semi-definite (PSD) and is called a (\emph{weighted}) \emph{graph Laplacian matrix}.

\begin{algorithm}[tb]
  \caption{RCMC method~\citep{Iwata2023,Sumiya2020-ti}}\label{alg:rcmc}
  \begin{algorithmic}[1]
    \Input{rate constant matrix $K \in \R^{V \times V}$, initial value $p \in \R^V$, $t_{\mathrm{max}} \in \Rpp$, $\mathrm{opt} \in \set{\texttt{full}, \texttt{last}}$} 
    \Output{$t^{(1)}, \dotsc, t^{(k)} \in \Rpp$ and $q^{(1)}, \dotsc, q^{(k)} \in \R^V$ for some $k \in [0, n)$}
    \State{$S^{(0)} \coloneqq \varnothing$, $T^{(0)} \coloneqq V$, $K^{(0)} \coloneqq K$}\label{line:rcmc-step1-begin}
    \For{$j = 1, \dotsc, n$}
        \State{$s \gets s^{(j)} \coloneqq \argmax\Set[\big]{-K^{(j-1)}_{vv}}{v \in T^{(j-1)}}$}\Comment{Ties are consistently broken}\label{line:rcmc-greedy}
        \If{$-K^{(j-1)}_{ss} < 1 / t_{\mathrm{max}}$}\label{line:rcmc-stop}
            \State{$k \coloneqq j-1$ and \textbf{break}}
        \EndIf
        \State{$S^{(j)} \coloneqq S^{(j-1)} \cup \set{s}$, $T^{(j)} \coloneqq T^{(j-1)} \setminus \set{s}$}\label{line:rcmc-add-S}
        \State{$K^{(j)}_{uv} \coloneqq K^{(j-1)}_{uv} - K^{(j-1)}_{us} K^{(j-1)}_{sv} / K^{(j-1)}_{ss}$ for $u, v \in T^{(j)}$} \Comment{Use~\eqref{eq:K-diag-correction} for diagonals}\label{line:rcmc-update-K}
    \EndFor
    \State{Compute $t^{(j)}$ and $q^{(j)}$ by~\eqref{eq:t-q} for $j \in [k]$ if $\mathrm{opt} = \texttt{full}$ and for $j = k$ if $\mathrm{opt} = \texttt{last}$}\label{line:compute-q}
  \end{algorithmic}
\end{algorithm}

Consider the master equation~\eqref{def:master} defined by a rate constant matrix $K$ with an initial condition $x(0) = p \in \R^n$ over a time interval $t \in [0, t_{\mathrm{max}}]$.
The \emph{rate constant matrix contraction} (RCMC) \emph{method}~\cite{Iwata2023,Sumiya2020-ti}, described in \cref{alg:rcmc}, computes an approximate solution to the master equation.
The RCMC method consists of the following two steps:
\begin{enumerate}[{label={\textbf{Step~\arabic*}:},ref={\arabic*}}]
    \item greedy selection of steady-states (Lines~\ref{line:rcmc-step1-begin}--\ref{line:rcmc-update-K}),\label[step]{step:1}
    \item computation of approximate solutions (\cref{line:compute-q}).\label[step]{step:2}
\end{enumerate}

To distinguish Lines~\ref{line:rcmc-step1-begin}--\ref{line:rcmc-update-K} in \cref{alg:rcmc} and other implementations of \cref{step:1} (\cref{alg:cholesky,alg:cholesky-crout,alg:lazy-fast}) that will be presented later, we refer to Lines~\ref{line:rcmc-step1-begin}--\ref{line:rcmc-update-K} as \textsc{Greedy}.
\textsc{Greedy} constructs a growing sequence of subsets $\emptyset = S^{(0)} \subsetneq S^{(1)} \subsetneq \dotsb \subsetneq S^{(k)} \subseteq V$ with $S^{(j)} = \set{s^{(1)}, \dotsc, s^{(j)}}$, where $k$ denotes the number of iterations.
Each element $s = s^{(j)}$ is called a \emph{steady-state} in expectation of $\dot{x}_s(t) \approx 0$ for $t \ge t^{(j)}$, where $t^{(j)}$ is the \emph{reference time} that will be computed in \cref{step:2} later.
Conversely, elements in $T^{(j)} \coloneqq V \setminus S^{(j)}$ are said to be \emph{transient}.
Starting with $S^{(0)} = \varnothing$, $T^{(0)} = V$ and $K^{(0)} \coloneqq K$, in the $j$th iteration, the algorithm takes the element $s^{(j)} \in T^{(j-1)}$ with the largest diagonal entry of $K^{(j-1)}$ in the absolute value at \cref{line:rcmc-greedy}.
Here, we assume that ties are consistently broken by choosing the smallest element with respect to a fixed total order on $V$, including all the subsequent algorithms.
Then, the algorithm moves $s^{(j)}$ from $T^{(j-1)}$ to $S^{(j-1)}$ (\cref{line:rcmc-add-S}) and modifies $K^{(j-1)}$ into $K^{(j)}$ by a rank-one update (\cref{line:rcmc-update-K}).
The matrix $K^{(j)}$ is the Schur complement of $S^{(j)}$ in $K = K^{(0)}$ (the original rate constant matrix), i.e.,
\begin{align}\label{eq:Kj-schur}
    K^{(j)}
    = K_{TT} - K_{TS} K_{SS}^{-1} K_{ST}
    = -(L_{TT} - L_{TS} L_{SS}^{-1} L_{ST}) \Pi_{TT}^{-1}
\end{align}
holds with $S = S^{(j)}$ and $T = T^{(j)}$.
Since the Schur complement $L_{TT} - L_{TS} L_{SS}^{-1} L_{ST}$ of $S$ in $L$ is again a graph Laplacian matrix,  $K^{(j)}$ is also a rate constant matrix, which is formally stated as follows.

\begin{proposition}[{\citep{Iwata2023}}]\label{prop:schur-complement-rcm}
    The matrix $K^{(j)} \in \R^{T^{(j)} \times T^{(j)}}$ satisfies the axioms~\ref{item:rcm1}--\ref{item:rcm3} of rate constant matrices on $T^{(j)}$.
\end{proposition}

The iteration of \textsc{Greedy} ends when all the diagonal entries in $-K^{(j)}$ are less than $1/t_{\mathrm{max}}$.

\cref{step:2} computes the reference time $t^{(j)} \in \Rpp$ and a vector $q^{(j)} \in \R^V $ according to the following formulas\footnote{%
  This formula for $t^{(j)}$, called ``diag'' in~\citep{Iwata2023}, is used in the original RCMC method by~\citet{Sumiya2015-br}.
  More accurate, albeit complicated, formulas are presented in~\citep{Iwata2023}.
}:
\begin{align}\label{eq:t-q}
    t^{(j)} \coloneqq -\frac{1}{K^{(j)}_{ss}}, \;
    q^{(j)} = \begin{pmatrix} q^{(j)}_S \\ q^{(j)}_T \end{pmatrix} \coloneqq \begin{pmatrix}
        K_{SS}^{-1}K_{ST}UK_{TS}K_{SS}^{-1} & -K_{SS}^{-1}K_{ST}U \\
        -UK_{TS}K_{SS}^{-1} & U
    \end{pmatrix} \begin{pmatrix}
        p_S \\ p_T
    \end{pmatrix},
\end{align}
where $s = s^{(j)}$, $S = S^{(j)}$, $T = T^{(j)}$, $K = K^{(0)}$, and $U \coloneqq \Diag\prn[\big]{\ones^\top + \ones^\top K_{TS} K_{SS}^{-2}K_{ST}}$.
The reference times $t^{(j)}$ satisfy $0 \le t^{(1)} \le \dotsb \le t^{(k)} \le \tmax$ and captures times when the solution $x(t)$ change significantly.
Every $q^{(j)}$ serves is a rough approximation of $x\prn[\big]{t^{(j)}}$; for the accuracy analysis, see~\citep[Section~4]{Iwata2023}.

An argument ``opt'' controls the set of $j$ for which $(t^{(j)}, q^{(j)})$ is computed in \cref{step:2}.
The algorithm computes the entire sequence $(t^{(0)}, q^{(0)}), \dotsc, (t^{(k)}, q^{(k)})$ if $\mathrm{opt} = \texttt{full}$ and only $(t^{(k)}, q^{(k)})$ if $\mathrm{opt} = \texttt{last}$.
The former is used when it is necessary to capture the entire trajectory $\Set{(t, x(t))}{0 \le t \le \tmax}$, while the latter is used when it is enough only to obtain the last value $x(\tmax)$, reducing the computational cost of \cref{step:2}.

\subsection{Efficiency and Numerical Stability of the RCMC Method}\label{sec:runtime-rcmc}

Here, we analyze the running time of the RCMC method.
For an efficient implementation, we assume that the input matrix $K$ is given as a sparse matrix, but $K^{(j)}$ is held as a dense matrix since we cannot expect $K^{(j)}$ to be sparse for larger $j$ due to fill-ins even if $K$ is sparse.
Then, the time complexity of \cref{alg:rcmc} is analyzed as follows.

\begin{table}[t]
    \centering
    \begin{tabular}{ccc}\toprule
        Sparsity of $K$ & opt & Running time \\\midrule
        Dense    & \texttt{full} & $\Theta(k^2 n)$ \\
        Dense    & \texttt{last} & $\Theta(kn)$    \\
        Sparse   & \texttt{full} & $\Theta(k^3) + \tilde\Theta(kn)$   \\
        Sparse   & \texttt{last} & $\Theta(k^2) + \tilde\Theta(n)$ \\
    \bottomrule\end{tabular}
    \caption{Running time of \cref{step:2}.}\label{tbl:runtime-step2}
\end{table}

\begin{theorem}\label{thm:rcmc-complexity}
    \textscup{Greedy} runs in $\Theta(kn^2)$ time and \cref{step:2} runs in time summarized in \cref{tbl:runtime-step2}.
\end{theorem}

\begin{proof}
    It is clear that the bottleneck of \textsc{Greedy} is \cref{line:rcmc-update-K}, which takes $\Theta(kn^2)$ time.
    We analyze the running time of \cref{step:2}.
    As described in \cite[Section~5.1]{Iwata2023}, the LU decomposition of $K_{SS}$ can be immediately obtained from the Schur complements of $K$.
    Since we can compute $q^{(j)}$ in $\Theta(j^2 + \nnz(K_{S^{(j)}T^{(j)}}))$ time for every $j \in [k]$, the computational costs of \cref{step:2} listed in \cref{tbl:runtime-step2} are obtained.
\end{proof}

From \cref{thm:rcmc-complexity} and \cref{tbl:runtime-step2}, the total running time of the RCMC method is $\Theta(kn^2)$ regardless of the sparsity of $K$ and the option.

\Citet[Section~5.2]{Iwata2023} gave the following observation on the numerical stability of the RCMC method.
By \cref{prop:schur-complement-rcm}, $K^{(j-1)}_{ss} < 0$ and $K^{(j-1)}_{us}, K^{(j-1)}_{sv} \ge 0$ hold for any $u, v \in T^{(j)}$.
Thus, the fraction $K^{(j-1)}_{us} K^{(j-1)}_{sv} / K^{(j-1)}_{ss}$ is non-positive.
As a result, the subtraction of $K^{(j-1)}_{us} K^{(j-1)}_{sv} / K^{(j-1)}_{ss}$ from $K^{(j-1)}_{uv}$ at \cref{line:rcmc-update-K} for $u \ne v$ is indeed the addition of two nonnegative numbers, ensuring the operation numerically stable.
On the other hand, in the case of $u = v$, both $K^{(j-1)}_{vv}$ and $K^{(j-1)}_{vs} K^{(j-1)}_{sv} / K^{(j-1)}_{ss}$ become non-positive.
Thus, their subtraction may cause a catastrophic cancellation.
For this problem, the following straightforward remedy was proposed in~\citep{Iwata2023}: first fill out the off-diagonal entries of $K^{(j)}$, and then, calculate every $v$th $(v \in T^{(j)})$ diagonal entry $K^{(j)}_{vv}$ as
\begin{equation}\label{eq:K-diag-correction}
  K^{(j)}_{vv} \gets -\sum_{u \in T^{(j)} \setminus \set{v}} K^{(j)}_{uv}.
\end{equation}
Since $K^{(j)}$ satisfies~\ref{item:rcm2},~\eqref{eq:K-diag-correction} computes the same value as \cref{line:rcmc-update-K},  without increasing the time complexity.
In addition, all the summands in~\eqref{eq:K-diag-correction} are nonnegative by~\ref{item:rcm1}.
Hence,~\eqref{eq:K-diag-correction} provides a numerically stable implementation of \cref{line:rcmc-update-K}.
\Cref{line:compute-q} is also free from subtractions of like-sign numbers~\citep[Section~5.2]{Iwata2023}.


\subsection{Greedy MAP Inference for DPPs}\label{sec:dpp-map-inference}
Let $L \in \R^{V \times V}$ be a positive semi-definite (PSD) matrix.
A \emph{determinantal point process} (DPP)~\citep{Macchi_1975} with a kernel matrix $L$ is a probability distribution on $2^V$ such that the probability of obtaining every $X \subseteq V$ is proportional to $\det L_{XX}$.
The \emph{maximum a posteriori} (MAP) \emph{inference} for DPPs (under the cardinality constraint) is the problem of finding $X \subseteq V$ with $|X| = k$ that maximizes $\det L_{XX}$, where $k \in [0, n]$ is a prescribed parameter.
This problem is equivalent to maximizing a function $f_L\colon 2^V \to \R \cup \set{-\infty}$ defined by
\begin{align}
  f_L(X) \coloneqq \log \det L_{XX} \quad (X \subseteq V),
\end{align}
where $\log 0$ is set to $-\infty$, under $|X| \le k$.
The function $f_L$ is known to be \emph{submodular}~\citep{fan1968inequality}.
Here, a function $f\colon 2^V \to \R \cup \set{-\infty}$ is called \emph{submodular} if $f(v \given X) \ge f(v \given Y)$ for all $X, Y \subseteq V$ with $X \subseteq Y$ and $v \in V \setminus Y$, where $f(v \given X) \coloneqq f(X \cup \set{v}) - f(X)$ denotes the \emph{marginal gain} of $v \in V \setminus X$ under $X \subseteq V$.

Maximizing a submodular function $f\colon 2^V \to \R \cup \set{-\infty}$ under the cardinality constraint $|X| \le k$ requires exponentially many queries of function values in general~\cite{nemhauser1978best}, and even the MAP inference for DPPs is NP-hard~\citep{Ko_1995}.
Thus, polynomial-time approximation algorithms are commonly used for this problem.
The standard \emph{greedy method} works as follows: starting from $S^{(0)} \coloneqq \varnothing$, in each $j$th ($j \in [k]$) iteration, take $s^{(j)} \coloneqq \argmax\Set[\big]{f(v \given S^{(j-1)})}{v \in T^{(j-1)}}$ with $T^{(j-1)} \coloneqq V \setminus S^{(j-1)}$ and let $S^{(j)} \coloneqq S^{(j-1)} \cup \set[big]{s^{(j)}}$.
The greedy method evaluates $f$ a total of $\Theta(kn)$ times, and enjoys a $(1 - 1/\mathrm{e})$-approximation guarantee if $f$ is \emph{monotone}, i.e., $f(X) \le f(Y)$ for $X \subseteq Y$~\citep{nemhauser1978analysis}.
Even if not, the greedy method is commonly adopted as a practical heuristic.

\subsection{Relationship between the RCMC Method and the Greedy DPP MAP Inference}\label{sec:relation}
\Citet{Iwata2023} indicated that \textsc{Greedy} can be regarded as a variant of the greedy DPP MAP inference as follows.
Let $K = -L\Pi^{-1} \in \R^{V \times V}$ be a rate constant matrix and $K^{(j-1)}$, $S^{(j-1)}$, and $T^{(j-1)}$ the values computed in \cref{alg:rcmc}.
At \cref{line:rcmc-greedy}, the $v$th $(v \in T^{(j-1)})$ diagonal entry of $K^{(j-1)}$ can be expressed as
\begin{align}\label{eq:K-diag}
  K^{(j-1)}_{vv}
  = \frac{\det K_{S \cup \set{v},S \cup \set{v}}}{\det K_{SS}}
  = -\frac{\det L_{S \cup \set{v}, S \cup \set{v}}}{\pi_v \det L_{SS}}
\end{align}
by~\citep[Appendix~E]{Iwata2023}, where $S = S^{(j-1)}$.
Define functions $f_{\Pi^{-1}}$ and $f_{-K}$ in the same way as $f_L$, i.e.,
\begin{align}
  f_{\Pi^{-1}}(X)
   & \coloneqq \log \det \prn[\big]{\Pi^{-1}}[X]
  = - \sum_{v \in X} \log \pi_v \quad (X \subseteq V),\\
  f_{-K}(X)
   & \coloneqq \log \det (-K)[X]
  = \log \det L_{XX} - \sum_{v \in X} \log \pi_v
  = f_{L}(X) + f_{\Pi^{-1}}(X) \quad (X \subseteq V).\label{def:f_minus_K}
\end{align}
Then, $f_{-K}$ is submodular since $f_L$ is submodular and $f_{\Pi^{-1}}$ is modular (linear).
By~\eqref{eq:K-diag} and~\eqref{def:f_minus_K}, we have the following proposition.

\begin{proposition}[{\citep[Appendix~E]{Iwata2023}}]\label{prop:marginal-schur}
    For $j \in [k]$ and $v \in T^{(j-1)}$,
    \begin{align}
        f_{-K}\prn[\big]{v \given S^{(j-1)}} = \log\prn[\big]{-K^{(j-1)}_{vv}}
    \end{align}
    holds.
\end{proposition}

\Cref{prop:marginal-schur} implies that \cref{step:1} of the RCMC method is essentially the greedy algorithm for maximizing the submodular function $f_{-K}$.
In particular, if $\pi = \ones$, it coincides with the greedy DPP MAP inference for a kernel matrix $L = -K$, in which only the stopping criterion is changed from $|S| = k$ to $-K^{(j-1)}_{ss} < 1/t_\mathrm{max}$.
\Cref{alg:rcmc} for general $K$ is the variant of the greedy DPP MAP inference with the objective submodular function modified by adding the modular function $f_{\Pi^{-1}}$.

\section{Accelerating the RCMC Method}\label{sec:fast-greedy}
As we described in \cref{sec:preliminaries}, the bottleneck of the RCMC method is \textsc{Greedy} for \cref{step:1}, which is essentially the same as the greedy DPP MAP inference.
This section describes how we can accelerate \cref{step:1} via existing faster implementations of the greedy DPP MAP inference.
In \cref{sec:cholesky-decomposition}, we introduce \textsc{FastGreedy}, which is an adaptation of the algorithm of \citet{Chen2018-nk}, through its connection to the Cholesky the decomposition.
In \cref{sec:lazy-greedy}, we describe a faster algorithm, \textsc{LazyFastGreedy}, presented by \citet{Hemmi2022-zj}.
We consider the time complexity and the numerical stability of \textsc{LazyFastGreedy} in \cref{sec:lazy-greedy-runtime}.

\subsection{FastGreedy via Cholesky Decomposition}\label{sec:cholesky-decomposition}

\begin{figure}[t]
    \begin{minipage}[t]{0.49\columnwidth}
        \begin{algorithm}[H]
              \caption{pivoted Gaussian elimination}\label{alg:cholesky}
              \begin{algorithmic}[1]
                \Input{A PSD matrix $L \in \R^{V \times V}$ and $\varepsilon \in \Rp$}
                \Output{A (partial) Cholesky factor $C$ of $L$}
                \State{$S^{(0)} \coloneqq \varnothing$, $T^{(0)} \coloneqq V$, $L^{(0)} \coloneqq L$}
                \For{$j = 1, \dotsc, n$}
                    \State{$s \gets s^{(j)} \coloneqq \argmax_{v \in T^{(j-1)}} L^{(j-1)}_{vv}$}\label{line:cholesky-argmax}
                    \If{$L^{(j-1)}_{ss} < \varepsilon$}\label{line:cholesky-stop}
                        \State{$k \coloneqq j-1$ and \textbf{break}}
                    \EndIf
                    \State{$S^{(j)} \coloneqq S^{(j-1)} \cup \set{s}$, $T^{(j)} \coloneqq T^{(j-1)} \setminus \set{s}$}
                    \State{$C_{sj} \coloneqq \sqrt{L^{(j-1)}_{ss}}$}
                    \State{$C_{uj} \coloneqq L^{(j-1)}_{us} / C_{sj}$ for $u \in T^{(j)}$}\label{line:C-uj-gaussian}
                    \State{$L^{(j)}_{uv} \coloneqq L^{(j-1)}_{uv} - L^{(j-1)}_{us} L^{(j-1)}_{sv} / L^{(j-1)}_{ss}$ for $u, v \in T^{(j)}$}
                \EndFor
              \end{algorithmic}
        \end{algorithm}
    \end{minipage}\hfill
    \begin{minipage}[t]{0.49\columnwidth}
        \begin{algorithm}[H]
              \caption{pivoted Doolittle's algorithm}\label{alg:cholesky-crout}
              \begin{algorithmic}[1]
                \Input{A PSD matrix $L \in \R^{V \times V}$ and $\varepsilon \in \Rp$}
                \Output{A (partial)  Cholesky factor $C$ of $L$}
                \State{$S^{(0)} \coloneqq \varnothing$, $T^{(0)} \coloneqq V$, $d^{(0)} \coloneqq \diag(L)$}
                \For{$j = 1, \dotsc, n$}
                    \State{$s \gets s^{(j)} \coloneqq \argmax_{v \in T^{(j-1)}} d^{(j-1)}_v$}\label{line:cholesky-crout-pivoting}
                    \If{$d^{(j-1)}_s < \varepsilon$}\label{line:cholesky-crout-stop}
                        \State{$k \coloneqq j-1$ and \textbf{break}}
                    \EndIf
                    \State{$S^{(j)} \coloneqq S^{(j-1)} \cup \set{s}$, $T^{(j)} \coloneqq T^{(j-1)} \setminus \set{s}$}
                    \State{$C_{sj} \coloneqq \sqrt{d^{(j-1)}_s}$}\label{line:cholesky-diag}
                    \State{$C_{uj} \coloneqq (L_{us} - \inpr{C_{uJ}, C_{sJ}}) / C_{sj}$ for $u \in T^{(j)}$, where $J \coloneqq [j-1]$}\label{line:cholesky-offdiag}
                    \State{$d^{(j)}_v \coloneqq d^{(j-1)}_v - {C_{vj}}^2$ for $v \in T^{(j)}$}\label{line:cholesky-d}
                \EndFor
              \end{algorithmic}
        \end{algorithm}
    \end{minipage}
\end{figure}

The (pivoted) \emph{Cholesky decomposition} of a PSD matrix $L \in \R^{V \times V}$ is a factorization of $L$ in the form of $L = CC^\top$, where $C = PG \in \R^{V \times n}$ is the product of a permutation matrix $P \in \R^{V \times n}$ and a lower-triangular matrix $G \in \R^{n \times n}$.
Note that rows and columns of $C$ are indexed by $V$ and $[n]$, respectively.
\Cref{alg:cholesky,alg:cholesky-crout} detail two methods for the Cholesky decomposition with a tolerance parameter $\varepsilon$.
Both algorithms fill entries of $C$ column-wise from left to right, treating unassigned entries of $C$ as $0$.
\Cref{alg:cholesky} is the standard Gaussian elimination with full pivoting.
We call $C$ computed by \cref{alg:cholesky} with $\varepsilon = 0$ the \emph{Cholesky factor} of $L$ with the order of pivots $s^{(1)}, \dotsc, s^{(k)}$, where $k = \rank L$ holds here.
In particular, the corresponding permutation matrix $P$ is given by $P_{s^{(j)}j} = 1$ for $j \in [k]$.
The intermediate matrix $L^{(j)}$ is the Schur complement of $S^{(j)}$ in $L$.
\cref{alg:cholesky-crout} keeps track of only the diagonal entries of $L^{(j)}$ as a vector $d^{(j)}$, without explicitly computing off-diagonals of $L^{(j)}$.
Note that $C_{uJ}$ and $C_{sJ}$ at \cref{line:cholesky-offdiag} are (row) vectors of dimension $|J| = j-1$ as $u$ and $s$ are single elements in $V$; hence their inner product is defined.
Whereas the unpivoted version of \cref{alg:cholesky-crout} originates from \citet{doolittle1878method} and \citet{gauss1809theoria}, the origin of the pivoted version is unclear; for a detailed history of the Cholesky decomposition, see~\citep{GRCAR2011163}.
Properties of \cref{alg:cholesky,alg:cholesky-crout} are summarized below.

\begin{theorem}\label{thm:equivalence-cholesky-crout}
    In \cref{alg:cholesky,alg:cholesky-crout}, the following hold:
    \begin{enumerate}
        \item both algorithms yield the same values for $k$, $s^{(j)}$ $(j \in [k+1])$, and $C_{uj}$ $(j \in [k],\, u \in T^{(j-1)})$,
        \item $d^{(j)}_v = L^{(j)}_{vv}$ holds for $j \in [0, k]$ and $v \in T^{(j)}$.
    \end{enumerate}
\end{theorem}

Although we believe that \cref{thm:equivalence-cholesky-crout} is well-known in the literature, we provide its proof as we could not find exact references.
For this, we first present the following lemma, which will also be used in the proof of \cref{lem:stable-d_v-fast}.

\begin{lemma}\label{lem:L-j-as-inner-product}
    Let $k$, $C$, $L^{(j)}$, and $T^{(j)}$ $(j \in [0, k])$ be the values computed by \cref{alg:cholesky} with $\varepsilon = 0$.
    Then, for $j \in [0, k]$ and $u,v \in T^{(j)}$, we have
    $L_{uv}^{(j)} = L_{uv} - \inpr{C_{uJ}, C_{vJ}}$ with $J \coloneqq [j]$.
\end{lemma}

\begin{proof}
    Let $T = T^{(j)}$ and $S = V \setminus T$.
    By $L = CC^\top$, we have
    \begin{gather}
        \begin{pmatrix}
            L_{SS} & L_{ST} \\
            L_{TS} & L_{TT}
        \end{pmatrix}
        = \begin{pmatrix}
            C_{SJ} & O \\
            C_{TJ} & *
        \end{pmatrix}
        \begin{pmatrix}
            C_{SJ}^\top & C_{TJ}^\top \\
            O & *
        \end{pmatrix}
        = \begin{pmatrix}
            C_{SJ}C_{SJ}^\top & C_{SJ} C_{TJ}^\top \\
            C_{TJ}C_{SJ}^\top & *
        \end{pmatrix},
    \end{gather}
    where $O$ denotes a zero matrix of appropriate size and $*$ means some matrix.
    Since $L^{(j)}$ is the Schur complement of $S$ in $L$, we have
    \begin{gather}
        L^{(j)}
        = L_{TT} - L_{TS}L_{SS}^{-1}L_{ST}
        = L_{TT} - C_{TJ}C_{SJ}^\top \prn[\big]{C_{SJ}^\top}^{-1} \prn[\big]{C_{SJ}}^{-1} C_{SJ} C_{TJ}^\top
        = L_{TT} - C_{TJ}C_{TJ}^\top
    \end{gather}
    as desired.
\end{proof}

\begin{proof}[{Proof of \cref{thm:equivalence-cholesky-crout}}]
    Let $k$, $C$, and $s^{(j)}$ $(j \in [k+1])$ denote the values computed by \cref{alg:cholesky}, and $\tilde k$, $\tilde C$, and $\tilde s^{(j)}$ $(j \in [\tilde k+1])$ those by \cref{alg:cholesky-crout}.
    Let $S^{(j)} \coloneqq \set{s^{(1)}, \dotsc, s^{(j)}}$ and $T^{(j)} \coloneqq V \setminus S^{(j)}$ for $j \in [0, k]$, $\tilde S^{(j)} \coloneqq \set{\tilde s^{(1)}, \dotsc, \tilde s^{(j)}}$ and $\tilde T^{(j)} \coloneqq V \setminus \tilde S^{(j)}$ for $j \in [0, \tilde k]$.
    By induction on $j \in [0, \min\set{k, \tilde k}]$, we show $S^{(j)} = \tilde S^{(j)}$ (hence $T^{(j)} = \tilde T^{(j)}$), $C_{ul} = \tilde C_{ul}$ for $l \in [j]$ and $u \in T^{(l-1)} = \tilde T^{(l-1)}$, and $d^{(j)}_v = L^{(j)}_{vv}$ for $v \in T^{(j)} = \tilde T^{(j)}$.
    The claims are clear when $j = 0$.
    For $j \ge 1$, suppose that the claims hold for $j-1$ and consider the case of $j$.
    By $d_v^{(j-1)} = L_{vv}^{(j-1)}$ for $v \in T^{(j-1)}$, we have $s \coloneqq s^{(j)} = \tilde s^{(j)}$ and hence the conditions $L_{ss}^{(j-1)} < \varepsilon$ and $d_s^{(j)} < \varepsilon$ are equivalent, meaning $k = \tilde k$ if the induction works correctly.
    If these conditions are met, we also have $S^{(j)} = \tilde S^{(j)}$ and $C_{sj} = \tilde C_{sj}$.
    For $u \in T^{(j)}$, by \cref{lem:L-j-as-inner-product} and the inductive assumptions, we have
    \begin{align}
        C_{uj}
        = \frac{L_{us}^{(j-1)}}{C_{sj}}
        = \frac{L_{us} - \inpr{C_{uJ}, C_{sJ}}}{C_{sj}}
        = \frac{L_{us} - \inpr{\tilde C_{uJ}, \tilde C_{sJ}}}{\tilde C_{sj}}
        = \tilde C_{uj},
    \end{align}
    where $J \coloneqq [j-1]$.
    Furthermore, for $v \in T^{(j)}$, we have
    \begin{align}
        L^{(j)}_{vv}
        = L^{(j-1)}_{vv} - \frac{{L^{(j-1)}_{vs}}^2}{L^{(j-1)}_{ss}}
        = L^{(j-1)}_{vv} - C_{vj}
        = d_v^{(j-1)} - \tilde C_{vj}
        = d_v^{(j)},
    \end{align}
    completing the proof.
\end{proof}

The matrix $L^{(j)}$ in \cref{alg:cholesky} is the Schur complement of $S^{(j)}$ in $L = L^{(0)}$.
Hence, if $L$ is a graph Laplacian matrix, the following holds similarly to \cref{prop:schur-complement-rcm}.

\begin{proposition}\label{prop:L-j-is-laplacian}
    If $L$ is a graph Laplacian matrix, so is $L^{(j)}$ for any $j \in [0, k]$.
\end{proposition}

\Cref{prop:L-j-is-laplacian} immediately implies the following properties on the Cholesky factor of a graph Laplacian matrix.

\begin{proposition}\label{prop:cholesky-factor-laplacian}
    Let $L$ be a graph Laplacian matrix and $C$ the Cholesky factor of $L$ computed by \cref{alg:cholesky,alg:cholesky-crout}.
    Then, $C_{uj} \le 0$ and $C_{s^{(j)}j} \ge 0$ hold for $j \in [k]$ and $u \in T^{(j)}$.
\end{proposition}
\begin{proof}
    The claims hold from \cref{prop:L-j-is-laplacian} and the equality $C_{uj} = L^{(j)}_{us} / \sqrt{L^{(j)}_{ss}}$ with $s = s^{(j)}$ for $u \in T^{(j)}$ from \cref{line:C-uj-gaussian} in \cref{alg:cholesky}.
\end{proof}

We get back to the RCMC method.
In case of $\pi = \ones$, \textsc{Greedy} for \cref{step:1} is the same as \cref{alg:cholesky} applied to $L = -K$ and $\varepsilon = 1/t_\mathrm{max}$ except \cref{line:cholesky-diag,line:cholesky-offdiag}.
Since \cref{step:1} with $\pi = \ones$ and the greedy DPP MAP inference are equivalent and \cref{alg:cholesky,alg:cholesky-crout} yield the same output by \cref{thm:equivalence-cholesky-crout}, \cref{alg:cholesky-crout} can be used as an alternative implementation of the greedy DPP MAP inference, which is nothing but the \emph{fast greedy} presented by \citet{Chen2018-nk}.
We can also employ \cref{alg:cholesky-crout} as an implementation for \cref{step:1} of the RCMC method with a general rate constant matrix $K = -L\Pi^{-1}$ by modifying \cref{alg:cholesky-crout} as follows:
\begin{itemize}
    \item replace \cref{line:cholesky-crout-pivoting} with $s \gets s^{(j)} \coloneqq \argmax\set[\big]{d^{(j-1)}_v / \pi_v \given v \in T^{(j-1)}}$, and
    \item change the stopping criterion at \cref{line:cholesky-crout-stop} to $d^{(j-1)}_s / \pi_s < 1/t_\mathrm{max}$.
\end{itemize}
This algorithm, which we refer to as \textsc{FastGreedy}, yields the same result as \textsc{Greedy} because of the equivalence of \textsc{Greedy} and \cref{alg:cholesky}, and that of \cref{alg:cholesky,alg:cholesky-crout} due to \cref{thm:equivalence-cholesky-crout}.
This improves the running time of \cref{step:1} from $\Theta(kn^2)$ to $\Theta(k^2n)$, summarized as follows.

\begin{theorem}\label{thm:complexity-cholesky-crout}
    \textscup{FastGreedy} computes the same values for $k$ and $s^{(j)}$ $(j \in [k+1])$ as \textscup{Greedy} in $\Theta(k^2n)$ time.
\end{theorem}

\begin{proof}
    The equivalence of the algorithms is explained above.
    As for the running time, we take $\Theta(j)$ time to compute $C_{uj}$ for each $j \in [k]$ and $u \in T^{(j)}$ at \cref{line:cholesky-offdiag} in \cref{alg:cholesky-crout}.
    Hence, \cref{line:cholesky-offdiag} takes $\Theta(k^2n)$ time in total.
\end{proof}

Note that \textsc{FastGreedy} does not impose extra computational cost in \cref{step:2} as the LU decomposition of $K_{SS}$ with $S = S^{(j)}$ is obtained from the Cholesky factor $C_{SJ}$ ($J \coloneqq [j]$) of $L_{SS}$ as $K_{SS} = -C_{SJ} \cdot C_{SJ}^\top \Pi_{SS}^{-1}$.
Since \cref{step:2} can be done in $\Ord(k^2n)$ time by \cref{tbl:runtime-step2}, the total running time of the RCMC method with \textsc{FastGreedy} is improved to $\Theta(k^2 n)$ regardless of the sparsity of $K$ and the option of \texttt{full} or \texttt{last}.

\subsection{LazyGreedy and LazyFastGreedy}\label{sec:lazy-greedy}

The \emph{lazy greedy}~\citep{Minoux_1978} is a practically faster implementation of the greedy method for maximizing a general submodular function $f\vcentcolon 2^V \to \R \cup \set{-\infty}$ under a cardinality constraint.
In addition to the current solution $S^{(j-1)} \subseteq V$, the lazy greedy maintains an upper bound $\rho_v$ on the marginal gain $f\prn[\big]{v \given S^{(j-1)}}$ for each $v \in T^{(j-1)} \coloneqq V \setminus S^{(j-1)}$.
The value of $\rho_v$ represents an ``old'' marginal gain of $v \in T^{(j-1)}$ in a sense that $\rho_v = f\prn[\big]{v \given S^{(b_v)}}$ holds for some $b_v \le j-1$, and the submodularity of $f$ guarantees $f\prn[\big]{v \given S^{(j-1)}} \le \rho_v$.
At each $j$th iteration, the lazy greedy picks $u \in T^{(j-1)}$ with the largest $\rho$ value and updates $\rho_u$ to the current marginal gain $f\prn[\big]{u \given S^{(j-1)}}$.
If the updated $\rho_u$ is still no less than $\rho_v$ for any $v \in T^{(j-1)}$, then we can conclude that $u$ has the largest marginal gain under $S^{(j-1)}$ without updating the others, hence we add $u$ to $S^{(j-1)}$ to construct $S^{(j)}$.
Otherwise, this procedure is repeated until an element with the largest marginal gain is found.
The upper bounds ${(\rho_v)}_{v \in T^{(j-1)}}$ are efficiently managed by using a priority queue implemented by a binary heap.
If the lazy update works the best, the method evaluates the objective function only $\Theta(n)$ times.
While the worst-case time complexity of the lazy greedy remains unchanged up to slight additional costs of heap operations, it empirically outperforms the na\"ive greedy method in many real instances~\citep{Hemmi2022-zj}.

\begin{algorithm}[tb]
  \caption{\textsc{LazyFastGreedy}}\label{alg:lazy-fast}
  \begin{algorithmic}[1]
    \Input{rate constant matrix $K = -L\Pi^{-1} \in \R^{V \times V}$, initial vector $p \in \R^V$, $t_{\mathrm{max}} \in \Rpp$} 
    \Output{$t^{(1)}, \dotsc, t^{(k)} \in \Rpp$ and $q^{(1)}, \dotsc, q^{(k)} \in \R^n$}
    \State{$j \gets 1$, $S^{(0)} \coloneqq \varnothing$, $T^{(0)} \coloneqq V$, $d^{(0)} \coloneqq \diag(L)$, $b_v \gets 0$, $\rho_v \gets d^{(0)}_{v} / \pi_v$ $(v \in V)$}\label{line:lazy-fast-begin}
    \While{true}\label{line:lazy-fast-while}
        \State{$u \gets \argmax\Set[\big]{\rho_v}{v \in T^{(j-1)}}$}\label{line:lazy-fast-heap}
        \For{$l = b_u+1, \dotsc, j-1$}\label{line:lazy-fast-inner-loop}
            \State{$C_{ul} \coloneqq (L_{us^{(l)}} - \inpr{C_{uJ}, C_{s^{(l)}J}}) / C_{s^{(l)}l}$ with $J \coloneqq [l-1]$}\label{line:lazy-fast-off-diagonal}
        \EndFor
        \State{$d^{(j-1)}_u \coloneqq d^{(b_u)}_u - \sum_{l=b_u+1}^{j-1} {C_{ul}}^2$}\Comment{\textsc{StableLazyFastGreedy} modifies this line}\label{line:lazy-fast-diagonal}
        \State{$b_u \gets j-1$ and $\rho_u \gets d_u^{(j-1)} / \pi_u$}\label{line:lazy-fast-update-b_u}
        \If{$\rho_u \ge \rho_v$ for all $v \in T^{(j-1)}$}
            \State{$s \gets s^{(j)} \coloneqq u$}
            \If{$\rho_u < 1 / t_{\mathrm{max}}$}
                \State{$k \coloneqq j-1$ and \textbf{break}}
            \EndIf
            \State{$S^{(j)} \coloneqq S^{(j-1)} \cup \set{s}$, $T^{(j)} \coloneqq T^{(j-1)} \setminus \set{s}$}\label{line:lazy-fast-S-T}
            \State{$C_{sj} \coloneqq \sqrt{d_s^{(j-1)}}$}\label{line:lazy-fast-C_sj}
            \State{$j \gets j + 1$}\label{line:lazy-fast-step1-end}
        \EndIf
    \EndWhile
  \end{algorithmic}
\end{algorithm}

We can apply the lazy greedy to the greedy DPP MAP inference and \cref{step:1} of the RCMC method since they are equivalent to the greedy maximization of the submodular functions $f_L$ and $f_{-K}$, respectively.
\textsc{FastGreedy} (\cref{alg:cholesky-crout}) and the lazy greedy were combined as the \emph{lazy fast greedy} by \citet{Hemmi2022-zj}, which can be easily modified to the setting of the RCMC method as we did for \cref{alg:cholesky-crout}.
\Cref{alg:lazy-fast} details this variant, which we call \textsc{LazyFastGreedy}.
The value $\rho_v$ for $v \in V$ is the exponential of an upper bound on the marginal gain $f_{-K}\prn[\big]{v \given S^{(j-1)}}$ and $b_v$ keeps track of when $\rho_v$ is last updated.
Specifically,
\begin{align}
    \rho_v
    = \exp\prn[\big]{f_{-K}\prn[\big]{v \given S^{(b_v)}}}
    = -K_{vv}^{(b_v)}
    = \frac{L_{vv}^{(b_v)}}{\pi_v}
    = \frac{d_v^{(b_v)}}{\pi_v}
\end{align}
holds from \cref{prop:marginal-schur,thm:equivalence-cholesky-crout}; see also~\cite[Proposition~2.1]{Hemmi2022-zj}.
The validity of \textsc{LazyFastGreedy} is summarized as follows.

\begin{theorem}\label{thm:lazy-fast-validity}
    \textscup{LazyFastGreedy} (\cref{alg:lazy-fast}) computes the same values for $k$, $s^{(j)}$ $(j \in [k+1])$ and $d_u^{(j)}, C_{uj}$ $(u \in V,\, j \in [b_u])$ as \textscup{Greedy} and \textscup{FastGreedy}, where $b_u \; (u \in V)$ denotes the value at the end of \cref{alg:lazy-fast}.
\end{theorem}

\begin{proof}
    The claim follows from \cref{thm:complexity-cholesky-crout} and the validity of the lazy greedy~\citep{Minoux_1978}.
\end{proof}

\subsection{Runtime and Numerical Stability of \textscup{LazyFastGreedy}}\label{sec:lazy-greedy-runtime}

We perform the running time analysis of \textsc{LazyFastGreedy}, modifying \cite[Theorem~3.2]{Hemmi2022-zj} to the setting of the RCMC method.
For $v \in V$, let $b_v$ be the value in \cref{alg:lazy-fast} at the end of the algorithm.
In addition, we let $c_j$ be the count of the \textbf{\upshape{while}} iterations of Lines~\ref{line:lazy-fast-while}--\ref{line:lazy-fast-step1-end} with $j \in [k+1]$.

\begin{lemma}\label{lem:bc}
    The values of $b_v$ and $c_j$ satisfy the following.
    \begin{enumerate}
        \item We have $b_{s^{(j)}} = j - 1$ for $j \in [k+1]$, $0 \le b_v \le k$ for $v \in T^{(k)}$, $c_1 = 1$, and $1 \le c_j \le n-j+1$ for $j = 2, \dotsc, k+1$.
        The lower and upper bounds on $b_v$ and $c_j$ are attained if the lazy heuristics works the best and worst, respectively.\label{item:bc1}
        \item The inequality $\sum_{j=1}^{k+1} c_j - 1 \le \sum_{v \in V} b_v \le \sum_{j=1}^k (j-1) c_j$ holds.\label{item:bc2}
    \end{enumerate}
\end{lemma}

\begin{proof}
    The statement~\ref{item:bc1} is clear from the algorithm.
    The first inequality in~\ref{item:bc2} follows from the facts that $\sum_{j=1}^{k+1} c_j$ is the total count of the \textbf{\upshape{while}} iterations, and every iteration increases some $b_v$ by at least one, except the first iteration with $j=1$.
    Similarly, the second inequality in~\ref{item:bc2} is due to the fact that each iteration with $j$ increases some $b_v$ by at most $j-1$.
\end{proof}

Let $\moffdiag$ be the total dimension of vectors for which the inner products are computed in \cref{line:lazy-fast-off-diagonal}.
We can represent $\moffdiag$ as
\begin{align}\label{def:moffdiag}
    \moffdiag = \sum_{v \in V} \sum_{l=1}^{b_v} (l-1)
    = \frac12 \sum_{v \in V} b_v(b_v - 1).
\end{align}
Substituting into~\eqref{def:moffdiag} the lower and upper bounds on $b_v$ provided by \cref{lem:bc}, we obtain the following exact bounds on $\moffdiag$:
\begin{align}\label{eq:mlazy-bound}
    \frac{k(k-1)(k+1)}{6} \le \moffdiag \le \frac{k(k-1)(3n-2k-2)}{6}.
\end{align}

\begin{theorem}[{see~\cite[Theorem~3.2]{Hemmi2022-zj}}]\label{thm:lazy-fast}
    \textscup{LazyFastGreedy} (\cref{alg:lazy-fast}) runs in $\Theta(\moffdiag + n) + \Ord\prn[\big]{\sum_{j=1}^{k+1} c_j \log n}$ time, which is $\Theta(k^3 + n) + \Ord(k \log n)$ and $\Theta(k^2n) + \Ord(kn \log n)$ if the lazy heuristic works the best and worst, respectively. 
\end{theorem}

\begin{proof}
    The initialization at \cref{line:lazy-fast-begin} takes $\Theta(n)$ time, including the heap construction.
    \Cref{line:lazy-fast-heap,line:lazy-fast-off-diagonal} take $\Ord\prn[\big]{\sum_{j=1}^{k+1} c_j \log n}$ and $\Theta(\moffdiag)$ time, respectively, and the other operations can be done within these times.
    The lower and upper bounds are obtained from \cref{lem:bc} and~\eqref{eq:mlazy-bound}.
\end{proof}

We will observe in \cref{sec:experiments} that \textsc{LazyFastGreedy} runs substantially faster than \textsc{FastGreedy} for real data.

We now consider the possibility of numerical cancellations in \textsc{LazyFastGreedy}.
Recall that $C_{uj} \le 0$ holds for $u \in V$ and $j \in [b_u]$ by \cref{prop:cholesky-factor-laplacian}, and we also have $d_v^{(j)} = L_{vv}^{(j)} \ge 0$ for $v \in V$ and $j \in [0, k]$ since $L^{(j)}$ is a graph Laplacian matrix.
Thus, the operation at \cref{line:lazy-fast-off-diagonal} in \cref{alg:lazy-fast}, which is the subtraction of $\inpr{C_{uJ}, C_{s^{(l)}J}}$ from $L_{us^{(l)}}$, is free from subtractions of like-sign numbers.
Unfortunately, in contrast, \cref{line:lazy-fast-diagonal} involves subtracting two non-negative numbers, making the algorithm numerically unstable.
In fact, we have confirmed that \textsc{LazyFastGreedy} with double-precision floating numbers output incorrect $s^{(j)}$ at some $j$ for real data due to catastrophic cancellations; see \cref{sec:experiments}.

\section{Stabilizing \textsc{LazyFastGreedy}}\label{sec:proposed}

This section presents a numerically stable implementation of \textsc{LazyFastGreedy}.
In \cref{subsec:stable-lazy-greedy}, we propose a stable subroutine used instead of \cref{line:lazy-fast-diagonal} in \cref{alg:lazy-fast}.
We then describe its faster implementation via segment trees in \cref{sec:segment-tree} and its running time analysis in \cref{sec:stable-lazy-runtime}.

\subsection{StableLazyFastGreedy}\label{subsec:stable-lazy-greedy}

In \cref{sec:rcmc}, we have stabilized the computation of $K_{vv}^{(j)} = -L_{vv}^{(j)}/\pi_j = -d_v^{(j)}/\pi_j$ in \textsc{Greedy} (\cref{line:rcmc-update-K} in \cref{alg:rcmc})  using the identity $\sum_{u \ne v} K^{(j)}_{uv} = 0$, which is equivalent to
\begin{align}\label{eq:stable-d_v}
    d^{(j)}_v
    = -\sum_{u \in T^{(j)} \setminus \set{v}} L_{uv}^{(j)}.
\end{align}
We cannot directly apply the same method for the computation of $d_v^{(j)}$ in \textsc{FastGreedy} (\cref{line:cholesky-d} in \cref{alg:cholesky-crout}) and of $d_u^{(j-1)}$ in \textsc{LazyFastGreedy} (\cref{line:lazy-fast-diagonal} in \cref{alg:lazy-fast}) because some off-diagonals of $L^{(j)}$ will not be computed in \textsc{(Lazy)FastGreedy}.

Our first step for stabilizing \textsc{(Lazy)FastGreedy} is to rewrite the right-hand side of~\eqref{eq:stable-d_v} using $L$ and $C$ as follows.

\begin{lemma}\label{lem:stable-d_v-fast}
    Let $j \in [k]$, $v \in T^{(j)}$, $T' \coloneqq T^{(j)} \setminus \set{v}$, and $J \coloneqq [j]$.
    Then, we have
    \begin{align}\label{eq:stable-d_v-fast}
        d_v^{(j)}
        = \inpr*{\sum_{u \in T'} C_{uJ}, C_{vJ}} - \sum_{u \in T'} L_{uv}.
    \end{align}
\end{lemma}

\begin{proof}
    By~\eqref{eq:stable-d_v} and \cref{lem:L-j-as-inner-product}, we have
    \begin{align}
        d_v^{(j)}
        = -\sum_{u \in T'} L_{uv}^{(j)}
        = -\sum_{u \in T'} (L_{uv} - \inpr{C_{uJ}, C_{vJ}})
        = \inpr*{\sum_{u \in T'} C_{uJ}, C_{vJ}} - \sum_{u \in T'} L_{uv}
    \end{align}
    as desired.
\end{proof}

The summation in the right-hand side of~\eqref{eq:stable-d_v-fast} avoids subtraction of like-sign numbers thanks to \cref{prop:cholesky-factor-laplacian}.
Hence, \cref{lem:stable-d_v-fast} can be used to stabilize \textsc{FastGreedy} because $C_{ul}$ have been settled for all $u \in T^{(j)}$ and $l \in [j-1]$ at \cref{line:cholesky-d} in \cref{alg:cholesky-crout}.
However, \textsc{FastGreedy} stabilized in this way is no longer faster than the RCMC method.
Furthermore, \textsc{LazyFastGreedy} does not compute some entries of $C$ due to lazy updates, hindering us from applying \cref{lem:stable-d_v-fast}.
Calculating the omitted entries, of course, would ruin the benefits derived from the lazy heuristic.

Our idea to alleviate this dilemma is to introduce a ``compressed'' graph Laplacian matrix defined below.
Let $j \in [k]$, $u \in T$, $S' = S^{(j-1)} \cup \set{u}$, and $T' = V \setminus S' = T^{(j-1)} \setminus \set{u}$.
Recall that $\ones$ denotes the all-one vector of an appropriate dimension.
The \emph{compression} of $L$ with respect to $u$ at the $j$th iteration is a $(j + 1) \times (j + 1)$ matrix $L^{[j, u]}$ is defined by
\[
    L^{[j, u]} = \begin{pNiceMatrix}[first-row,first-col]
          & S'        & \compress  \\
        S'         & L_{S'S'} & L_{S'T'}\ones \\
        \compress  & \ones^\top L_{T'S'} & \ones^\top L_{T'T'} \ones
    \end{pNiceMatrix},
\]
where the rows and columns of $L^{[j, u]}$ are indexed by a set $S' \cup \set{\compress}$ with a new element $\compress$.
Namely, $L^{[j, u]}$ is a matrix obtained by aggregating rows and columns belonging to $T'$ into a single row and column, respectively, indexed by $\compress $.
It is easily checked that $L^{[j, u]}$ is again a graph Laplacian matrix.
We denote by $C^{[j, u]}$ the Cholesky factor of $L^{[j, u]}$ with the pivoting order $s^{(1)}, \dotsc, s^{(j-1)}, u$.
The following easy observation on $C^{[j, u]}$ plays an important role.

\begin{lemma}\label{lem:compressed}
    For $j \in [k]$, $u \in T^{(j-1)}$, and $l \in [j-1]$, we have $C^{[j,u]}_{\compress l} = \ones^\top C_{T' l}$, where $T' = T^{(j-1)} \setminus \set{u}$.
\end{lemma}

\begin{proof}
    Immediately follows from \cref{line:C-uj-gaussian} in \cref{alg:cholesky}.
\end{proof}


By \cref{lem:stable-d_v-fast,lem:compressed}, $d_u^{(j-1)}$ can be expressed by
\begin{align}\label{eq:stable-d_v-2}
    d_u^{(j-1)} = \inpr[\big]{C^{[j,u]}_{\compress J}, C_{uJ}} - L^{[j,u]}_{\compress u}
\end{align}
with $J = [j-1]$.
Furthermore, as with \cref{alg:cholesky-crout}, for $l \in [j-1]$, the $(\compress, l)$ entry $C^{[j, u]}_{\compress l}$ can be computed from $C^{[j,u]}_{\compress 1}, \dotsc, C^{[j,u]}_{\compress ,l-1}$ as
\begin{align}\label{eq:compressed-cholesky}
    C^{[j,u]}_{\compress l}
    = \prn[\big]{L^{[j,u]}_{\compress s^{(l)}} - \inpr[\big]{C^{[j,u]}_{\compress J}, C^{[j,u]}_{s^{(l)}J}}} / C^{[j,u]}_{s^{(l)}l}
    = \prn[\big]{L^{[j,u]}_{\compress s^{(l)}} - \inpr[\big]{C^{[j,u]}_{\compress J}, C_{s^{(l)}J}}} / C_{s^{(l)}l}
\end{align}
with $J = [l-1]$.

\begin{algorithm}[tb]
  \caption{Stable computation of $d_u^{(j-1)}$}\label{alg:stably-compute-diagonal}
  \begin{algorithmic}[1]
    \Input{$j \in [k]$, $u \in T^{(j-1)}$}
    \Output{$d_u^{(j-1)}$}
    \Function{StablyComputeDiagonal}{$j, u$}
        \For{$l = 1, \dotsc, j-1$}
            \State{$C_{\compress l}^{[j,u]} \coloneqq \prn[\big]{
              L^{[j,u]}_{\compress s^{(l)}} - \inpr[\big]{C^{[j,u]}_{\compress J}, C_{s^{(l)}J}}
              }/ C_{s^{(l)}l}$ with $J \coloneqq [l-1]$}\label{line:inner-product}
        \EndFor
        \State{\Return $\inpr{C^{[j,u]}_{\compress J}, C_{uJ}} - L^{[j,u]}_{\compress u}$ with $J \coloneqq [j-1]$}\label{line:inner-product-2}
    \EndFunction
  \end{algorithmic}
\end{algorithm}

Based on~\eqref{eq:stable-d_v-2} and~\eqref{eq:compressed-cholesky}, we describe in \cref{alg:stably-compute-diagonal} a procedure, \textsc{StablyComputeDiagonal}$(j, u)$, to compute $d_u^{(j-1)}$ via $C^{[j,u]}$.
Since $C^{[j,u]}$ is the Cholesky factor of a graph Laplacian matrix $L^{[j,u]}$, we have $C^{[j,u]}_{\compress l} \le 0$ for $l \in [j-1]$ by \cref{prop:cholesky-factor-laplacian}, and thus \cref{alg:stably-compute-diagonal} avoids subtractions of like-sign numbers.
\Cref{alg:lazy-fast} with \cref{line:lazy-fast-diagonal} being replaced by \textsc{StablyComputeDiagonal}$(j, u)$ is refereed to as \textsc{StableLazyFastGreedy}.

\subsection{Efficient Implementation by Segment Trees}\label{sec:segment-tree}

Naively implementing \cref{alg:stably-compute-diagonal} results in the running time of \textsc{StablyComputeDiagonal}$(j, \cdot)$ being $\Theta(j^2 + \nnz(L_{S'T'}))$.
This is because the inner product calculation at \cref{line:inner-product} takes $\Theta(j^2)$ time in total, and computing $L_{\compress v}^{[j,u]} = \ones^\top L_{T'v}$ for all $v \in S'$ requires $\Theta(\sum_{v \in S'} \nnz(L_{vT'})) = \Theta(\nnz(L_{S'T'}))$ time.
The $\Theta(\nnz(L_{S'T'}))$ term becomes a bottleneck if $L$ is dense or $j$ is small.
To reduce this complexity while maintaining numerical stability, we employ a data structure known as a \emph{segment tree} or \emph{range tree}~\cite{Chazelle1988-zw}.
For ease of exposition, we assume $V = [n]$ in this section.

We explain segment trees.
Let $(X, \oplus)$ be a semigroup; that is, $\oplus: X \times X \to X$ is a binary operation on a set $X$ satisfying the associative low $a \oplus (b \oplus c) = (a \oplus b) \oplus c$ for all $a,b,c \in X$.
As a computational model, assume that every element in $X$ can be stored in a unit cell of memory, and the operation $\oplus$ can be performed in constant time.
A segment tree $\mathcal{T}$ over $X$ is a data structure that manages an array $(a_1, \dotsc, a_n)$ of $n$ elements $a_1, \dotsc, a_n \in X$ and supports the following operations:
\begin{itemize}
    \item $\mathcal{T}.\textsc{init}(a_1, \dotsc, a_n)$: initialize a new segment tree $\mathcal{T}$ by $a_1, \dotsc, a_n \in X$,
    \item $\mathcal{T}.\textsc{sum}(i, j)$: compute $a_i \oplus a_{i+1} \oplus \dotsb \oplus a_j$ $(1 \le i \le j \le n)$,
    \item $\mathcal{T}.\textsc{update}(i, a)$: update $a_i$ to $a$ $(i \in [n], a \in X)$.
\end{itemize}
$\mathcal{T}$.\textsc{init}$(a_1, \dotsc, a_n)$ takes $\Theta(n)$ time and the other two operations can be done in $\Ord(\log n)$ time.

We employ segment trees in \textsc{StableLazyFastGreedy} by regarding $\Rp$ as a semigroup $(\Rp, +)$.
We maintain $n$ segment trees $\mathcal{T}_v$ $(v \in V)$, each of which is associated with a column $v$ of $L$.
At every beginning of \textbf{while} loop in \cref{alg:lazy-fast}, we ensure that
\begin{align}\label{eq:segtree-invariant}
    \mathcal{T}_v.\textsc{sum}(1, n) = \begin{cases}
        \sum_{u \in T^{(j-1)}} L_{uv} & (v \in S^{(j-1)}), \\
        \sum_{u \in T^{(b_v)} \setminus \set{v}} L_{uv} & (v \in T^{(j-1)}).
    \end{cases}
\end{align}
First, each $\mathcal{T}_v$ is initialized by the $v$th column of $L$ with the diagonal entry $L_{vv}$ replaced with $0$, i.e.,
\begin{align}
    \mathcal{T}_v.\textsc{init}(L_{1v}, \dotsc, L_{{v-1},{v}}, 0, L_{v+1,v}, \dotsc, L_{nv}),
\end{align}
which clearly satisfies~\eqref{eq:segtree-invariant} as $b_v = 0$ and $T^{(0)} = V$.
When $b_u$ is updated from $b_u^\mathrm{old}$ to $j-1$ at \cref{line:lazy-fast-update-b_u}, we execute $\mathcal{T}_u.\textsc{update}(s^{(l)}, 0)$ for $l = b_u^\mathrm{old} + 1, \dotsc, j-1$.
In addition, when $s \in V$ is removed from $T^{(j-1)}$ at \cref{line:lazy-fast-S-T}, we execute $\mathcal{T}_v.\textsc{update}(s, 0)$ for all $v \in S^{(j-1)}$.
These two updates keep~\eqref{eq:segtree-invariant} to hold.

By the invariant condition~\eqref{eq:segtree-invariant}, for any $u \in T^{(j-1)}$ with $b_u = j-1$ and $v \in S^{(j-1)} \cup \set{u}$, we have
\begin{align}
    L_{\compress v}^{[j,u]}
    = \sum_{w \in T^{(j-1)} \setminus \set{u}} L_{wv}
    = \mathcal{T}_v.\textsc{sum}(1, n) - L_{uv}
    = \mathcal{T}_v.\textsc{sum}(1, u-1) + \mathcal{T}_v.\textsc{sum}(u+1, n).
\end{align}
Hence, we can replace every computation of $L_{\compress v}^{[j,u]}$ at \cref{line:inner-product,line:inner-product-2} in \cref{alg:stably-compute-diagonal} with the twice calls of $\mathcal{T}_v$.\textsc{sum}, which cost $\Ord(\log n)$ time.
Furthermore, this computation avoids catastrophic cancellations since the segment tree performs only the operation of the semigroup $(\Rp, +)$.

For further speeding up on sparse matrices, we can restrict the entries of $\mathcal{T}_v$ to only the non-zeros in the $v$th column $L_{Vv}$ of $L$ to improve $n$ to $\nnz(L_{Vv}) = \nnz(K_{Vv})$ in the query complexity.

\subsection{Compelxity Analysis for StableLazyFastGreedy}\label{sec:stable-lazy-runtime}

We estimate the complexity of \textsc{StableLazyFastGreedy}.
The value $\moffdiag$, which is the total dimension of vectors for which the inner products are computed at \cref{line:lazy-fast-off-diagonal} in \cref{alg:lazy-fast}, is still given by~\eqref{def:moffdiag} in the same way as \textsc{LazyFastGreedy}.
In addition to $\moffdiag$, the runtime of \textsc{StableLazyFastGreedy} also depends on the total dimension, denoted by $\mdiag$, of vectors for which the inner products are computed at \cref{line:inner-product,line:inner-product-2} in \cref{alg:stably-compute-diagonal}.
Recall from \cref{sec:lazy-greedy} that $c_j$ denotes the count of \textbf{while} iterations with $j$, which is the same as the number of calls of \textsc{StablyComputeDiagonal}$(j, \cdot)$.
Using $c_j$, we can represent $\mdiag$ as
\begin{align}\label{eq:mdiag-c_j}
    \mdiag = \sum_{j=1}^{k+1} c_j
    \prn*{\sum_{l=1}^{j-1} (l-1) + j-1}
    = \frac12 \sum_{j=1}^{k+1} j (j-1) c_j.
\end{align}
By \cref{lem:bc}, we have
\begin{align}\label{eq:mrelax-bound}
    \frac{k(k+1)(k+2)}{6} \le \mdiag \le \frac{k(k+1)(k+2)(4n-3k-1)}{24}
\end{align}
and the lower and upper bounds are attained if the lazy heuristics work the best and worst, respectively.

The running time of \textsc{StableLazyFastGreedy} can be described by using $\mdiag$ and $\moffdiag$ as follows.
 
\begin{theorem}\label{thm:running-time-stable-lazy-fast}
    \textscup{StableLazyFastGreedy} runs in $\Theta(\moffdiag + \mdiag + \nnz(K)) + \Ord\prn[\big]{\sum_{j=1}^{k+1} j c_j \log n}$ time.
    If the lazy heuristic works the best and worst, the running time becomes $\Theta(k^3 + \nnz(K)) + \Ord(k^2 \log n)$ and $\Theta(k^3 n + \nnz(K)) + \Ord(k^2 n \log n)$, respectively.
\end{theorem}

\begin{proof}
    The most costly parts of \textsc{StableLazyFastGreedy} are \cref{line:lazy-fast-off-diagonal} in \cref{alg:lazy-fast} and \cref{line:inner-product} (and~\ref{line:inner-product-2}) in \cref{alg:stably-compute-diagonal}, which run in $\Theta(\moffdiag)$ and $\Theta(\mdiag)$ time, respectively.
    We estimate complexities on queries for segment trees.
    The initialization $\mathcal{T}_v.\textsc{init}$ for all $v \in V$ takes $\Theta(\sum_{v \in V} \nnz(K_{Vv})) = \Theta(\nnz(K))$ time.
    The sum query $\mathcal{T}_v.\textsc{sum}$ takes $\Ord(j \log n)$ time in every invocations of \textsc{StablyComputeDiagonal}$(j, \cdot)$, hence $\Ord\prn[\big]{\sum_{j=1}^{k+1} j c_j \log n}$ time in total.
    The update query $\mathcal{T}_v.\textsc{update}$ is called $k$ times for $s \in S^{(k)}$ and $b_v$ times for $v \in T^{(k)}$, meaning that the total computational time taken by the update queries is $\Ord((k^2 + \sum_{v \in V} b_v) \log n))$ and is $\Ord(\sum_{j=1}^{k+1} jc_j \log n)$ by \cref{lem:bc}.
    The number of heap queries at \cref{line:lazy-fast-heap} is less than that of segment trees because every execution of \cref{line:lazy-fast-heap} is followed by the call of \textsc{StablyComputeDiagonal} at \cref{line:lazy-fast-diagonal}.
    the best and worst time complexities are obtained from~\eqref{eq:mlazy-bound} and~\eqref{eq:mrelax-bound}, and by substituting $c_j = 1$ and $c_j = n-j+1$, respectively, for $j \in [k+1]$.
\end{proof}

\section{Relaxing StableLazyFastGreedy}\label{sec:rslfg}
\subsection{Acceleration by Relaxing Stabilization}\label{sec:modified-stable-lazy-greedy}

If the lazy heuristic is effective, \textsc{StableLazyFastGreedy} runs in $\Theta(k^3)$ time by \cref{thm:running-time-stable-lazy-fast}, assuming $k = \Omega(\log n)$.
This performance is comparable to that of \textsc{LazyFastGreedy}.
However, if the lazy heuristic performs poorly, the running time increases significantly to $\Theta(k^3 n)$.
The bottleneck is the inner product calculation at \cref{line:inner-product} in \cref{alg:stably-compute-diagonal}.
To speed up the computation, we introduce a further heuristic based on the following error analysis on catastrophic cancellations.

Suppose that a computer stores an approximation $\hat{a} \in \R$ of a non-zero real number $a \in \R$, where the error arises from previous computations using finite-precision arithmetic.
The \emph{relative error} of $\hat{a}$ with respect to $a$ is defined as $\err_a(\hat{a}) \coloneqq \abs{a - \hat{a}} / \abs{a}$.
It is easily checked that for two positive reals $a$ and $b$ with $a > b$ and their approximations $\hat{a}$ and $\hat{b}$, the relative error of $\hat{x} \coloneqq \hat{a} - \hat{b}$ with respect to $x \coloneqq a - b$ is bounded as 
\begin{align}
    \err_x(\hat{x})
    \le \frac{a+b}{a-b} \max\set{\err_a(\hat{a}), \err_b(\hat{b})}
    = \frac{1+r}{1-r} \max\set{\err_a(\hat{a}), \err_b(\hat{b})},
\end{align}
where $r \coloneqq b/a$~\cite[Section~1.7]{Higham2002-zq}.
Therefore, if $r \le \varepsilon /(2 + \varepsilon)$ for some error tolerance parameter $\varepsilon \ge 0$, then $\err_x(\hat{x})$ is bounded by $(1 + \varepsilon) \max\set{\err_a(\hat{a}), \err_b(\hat{b})}$.
This implies that if an increase in relative error by a factor of at most $1 + 10^{-6}$ is acceptable, for example, we can subtract two numbers whose ratio is at most $10^{-6} / (2 + 10^{-6}) \approx 5 \times 10^{-7}$.
Although this requirement may seem hardly attainable, it is often met in rate constant matrices of chemical reaction networks.
However, since algorithms only have the approximations $\hat{a}$ and $\hat{b}$, but their true values, the ratio $b / a$ is not directly computable. 
The following theorem shows that the same error bound holds even if $\hat{b} / \hat{a}$ is used instead of $b / a$ under the mild assumption that $1 + \varepsilon \le {\max\set{\err_a(\hat a), \err_b(\hat b)}}^{-1}$.

\begin{theorem}\label{thm:relative-error}
    Let $a,\hat{a},b,\hat{b} \in \R$ be non-zero reals of the same sign with $|a| > |b|$, $|\hat{a}| > |\hat{b}|$, $x \coloneqq a - b$, $\hat{x} \coloneqq \hat{a} - \hat{b}$, and $e \coloneqq \max\set{\err_a(\hat a), \err_b(\hat b)}$.
    For $\varepsilon \in [0, e^{-1} - 1]$, if $\hat{b}/\hat{a} \le \varepsilon / (2 + \varepsilon)$, then $\err_x(\hat{x}) \le (1 + \varepsilon)e$ holds.
\end{theorem}

\begin{proof}
    We can assume $e \le 1$ because the claim is vacuously true for $e > 1$ as $[0, e^{-1} - 1]$ is empty.
    We further assume that $a,b,\hat a$, and $\hat b$ are positive without loss of generality.
    Let $\delta_a$ and $\delta_b$ be the values such that $\hat{a} = (1+\delta_a)a$ and $\hat{b} = (1+\delta_b)b$, respectively, and $r \coloneqq \hat{b} / \hat{a}$.
    Note that $e = \max\set{\abs{\delta_a}, \abs{\delta_b}}$ holds by $\hat{a}, \hat{b} > 0$ and $r < 1$ follows from $\hat a > \hat b$.
    The relative error of $\hat x$ with respect to $x$ is written as
    \begin{align}\label{eq:rewrite-err}
        \err_x(\hat x)
        = \frac{\abs{x - \hat{x}}}{x}
        = \frac{\abs{a\delta_a + b\delta_b}}{a - b}
        = \frac{\abs{\frac{\hat a \delta_a}{1+\delta_a} + \frac{\hat b \delta_b}{1+\delta_b}}}{\frac{\hat a}{1+\delta_a} - \frac{\hat b}{1+\delta_b}}
        = \frac{\abs{\frac{\delta_a}{1+\delta_a} + \frac{\delta_b}{1+\delta_b}r}}{\frac{1}{1+\delta_a} - \frac{1}{1+\delta_b}r}
        = \abs{f_r(\delta_a, \delta_b)},
    \end{align}
    where $f_r(p, q)$ is a function defined by
    \begin{align}
        f_r(p, q)
        \coloneqq \frac{\frac{p}{1+p} + \frac{q}{1+q}r}{\frac{1}{1+p} - \frac{1}{1+q}r} \quad \prn*{-e \le p,q \le +e, \, \frac{1}{1+p} > \frac{1}{1+q}r}.
    \end{align}
    Note that the domain of $f_r$ can be restricted to $(p,q)$ with $1/(1+p) > r/(1+q)$ because $1/(1+\delta_a) > r/(1+\delta_b)$ is implied by $a > b$.

    In the following, we show that $f_r(p, q)$ is monotone non-decreasing with respect to $p$ and $q$ for fixed $q$ and $p$, respectively.
    In fact, since $(p, q) = (+e, +e)$ and $(-e, -e)$ are in the domain of $f_r(p, q)$, the monotonicity gives lower and upper bounds on $f_r(\delta_a, \delta_b)$ as follows:
    \begin{align}
        -(1 + \varepsilon) e
        \le -\frac{1+r}{1-r} e
        =
        f_r(-e, -e) \le
        f_r(\delta_a, \delta_b) \le f_r(+e, +e) 
        = +\frac{1+r}{1-r} e
        \le +(1 + \varepsilon) e,
    \end{align}
    where the first inequality is obtained from $r \le \varepsilon/(2+\varepsilon)$, showing the claim of the theorem.
    
    Reformulating the expression, we have
    \begin{align}\label{eq:f_eps_positive}
        f_r(p, q)
        = \frac{1 - \frac{1-q}{1+q}r}{\frac{1}{1+p} - \frac{1}{1+q}r} - 1
        = \frac{r - \frac{1-p}{1+p}}{\frac{1}{1+p} - \frac{1}{1+q}r} + 1.
    \end{align}
    Therefore, $f_r(p, q)$ is monotone non-decreasing with respect to $p$ and $q$ if and only if $1 - \frac{1-q}{1+q}r \ge 0$ and $r - \frac{1-p}{1+p} \le 0$, respectively.
    These inequalities follow from
    \begin{align}
        r
        \le \frac{\varepsilon}{2 + \varepsilon}
        \le \frac{1-e}{1+e}
        \le \min\set*{\frac{1-p}{1+p}, \frac{1+q}{1-q}},
    \end{align}
    where the second and third inequalities are from the facts that a function $u \mapsto (1-u)/(1+u)$ is monotone decreasing and ${(1 + \varepsilon)}^{-1} \ge e \ge +p, -q$.
    Hence, the monotonicity of $f_r(p, q)$ has been proved.
\end{proof}

We incorporate \cref{thm:relative-error} into \cref{line:inner-product} in \cref{alg:stably-compute-diagonal} using the following relationship:
\begin{align}
    C_{\compress l}^{[j,u]} = C_{\compress l}^{[j-1,s^{(j-1)}]} - C_{ul},\label{eq:relax}
\end{align}
which comes from $C_{\compress l}^{[j,u]} = \sum_{v \in T^{(j-1)} \setminus \set{u}} C_{vl}$ and $C_{\compress l}^{[j-1,s^{(j-1)}]} = \sum_{v \in T^{(j-1)}} C_{vl}$.
Note that $C_{\compress l}^{[j-1,s^{(j-1)}]}$ is readily available since it is computed in \textsc{StablyComputeDiagonal}$(j-1, s^{(j-1)})$, which is previously invoked.
Thus, according to \cref{thm:relative-error}, we can obtain $C_{\compress l}^{[j,u]}$ using~\eqref{eq:relax} when $C_{ul}/C_{\compress l}^{[j-1,s^{(j-1)}]} \le \varepsilon/(2 + \varepsilon)$ with a prescribed error tolerance $\varepsilon$, avoiding the computation of the inner product.
We term this heuristic the \emph{relaxing heuristic} and the improved algorithm as \textsc{RelaxedStableLazyFastGreedy}.

\subsection{Complexity Analysis for RelaxedStableLazyFastGreedy}

We finally analyze the running time of \textsc{RelaxedStableLazyFastGreedy}.
Recall that $\moffdiag$ and $\mdiag$ denote the total dimensions of vectors for which the inner products are computed at \cref{line:lazy-fast-off-diagonal} in \cref{alg:lazy-fast} and \cref{line:inner-product},~\ref{line:inner-product-2} in \cref{alg:stably-compute-diagonal}, respectively.
Since the difference between \textsc{RelaxedStableLazyFastGreedy} and \textsc{StableLazyFastGreedy} is only \cref{line:inner-product} in \cref{alg:stably-compute-diagonal}, the running time estimated in \cref{thm:running-time-stable-lazy-fast} is still valid for \textsc{RelaxedStableLazyFastGreedy}; that is, \textsc{RelaxedStableLazyFastGreedy} runs in $\Theta(\moffdiag + \mdiag + \nnz(K)) + \Ord\prn[\big]{\sum_{j=1}^{k+1} j c_j \log n}$ time.
Here, the equality~\eqref{def:moffdiag} for $\moffdiag$ holds in both \textsc{StableLazyFastGreedy} and \textsc{RelaxedStableLazyFastGreedy}, whereas the relaxing heuristics improves the equality~\eqref{eq:mdiag-c_j} for $\mdiag$ into the following inequality.

\begin{lemma}\label{lem:mrelax-bound}
  In \textscup{RelaxedStableLazyFastGreedy}, $\mdiag$ satisfies
  \begin{align}
      \sum_{j=1}^{k+1} (j-1) c_j \le
      \mdiag \le \frac12 \sum_{j=1}^{k+1} j (j-1) c_j
  \end{align}
  and the lower and upper bounds are attained if the relaxing heuristics work the best and worst, respectively.
  In particular, $\mdiag$ is equal to
  \begin{itemize}
      \item $k(k+1)/2 = \Theta(k^2)$ if both the lazy and relaxing heuristics work the best,
      \item $k(k+1)(3n-2k-1)/6 = \Theta(k^2n)$ if the lazy and relaxing heuristics work the worst and best, respectively,
      \item $k(k+1)(k+2)/6 = \Theta(k^3)$ if the lazy and relaxing heuristics work the best and worst, respectively, and
      \item $k(k+1)(k+2)(4n-3k-1)/24 = \Theta(k^3n)$ if both the lazy and relaxing heuristics work the worst.
  \end{itemize}
\end{lemma}

\begin{proof}
    If the relaxing heuristic works perfectly, \cref{line:inner-product} in \cref{alg:stably-compute-diagonal} does not compute an inner product and \cref{line:inner-product-2} contributes to $\mdiag$ by $j-1$, resulting in $\mdiag = \sum_{j=1}^{k+1} (j-1)c_j$.
    If the relaxing heuristic does not work at all, \textsc{RelaxedStableLazyFastGreedy} is the same as \textsc{StableLazyFastGreedy} and~\eqref{eq:mdiag-c_j} holds.
    The latter part of the claim is obtained by substituting the bounds on $c_j$ given in \cref{lem:bc}.
\end{proof}

\cref{lem:mrelax-bound} leads us to the following time complexity on \textsc{RelaxedStableLazyFastGreedy}.

\begin{theorem}\label{thm:running-time-relaxed-stable-lazy-fast}
  The running time of \textscup{RelaxedStableLazyFastGreedy} is $\Theta(\moffdiag + \mdiag + \nnz(K)) + \Ord\prn[\big]{\sum_{j=1}^{k+1} j c_j \log n}$, which boils down to
  \begin{itemize}
      \item $\Theta(k^3 + \nnz(K)) + \Ord(k^2 \log n)$ if the lazy heuristics works the best,
      \item $\Theta(k^2n + \nnz(K)) + \Ord(k^2 n \log n)$ if the lazy and relaxing heuristics work the worst and best, respectively, and
      \item $\Theta(k^3n + \nnz(K)) + \Ord(k^2 n \log n)$ if both the lazy and relaxing heuristics work the worst.
  \end{itemize}
\end{theorem}

Even if the lazy heuristics works poorly, the relaxing heuristics can improve the dominating term in the complexity from $\Theta(k^3 n)$ up to $\Theta(k^2n)$.

\section{Numerical Experiments}\label{sec:experiments}

In this section, we show experimental results on the running time and numerical stability of the five algorithms: \textsc{Greedy}, \textsc{FastGreedy}, \textsc{LazyFastGreedy}, \textsc{StableLazyFastGreedy}, and \textsc{RelaxedStableLazyFastGreedy} (abbreviated as \textsc{RSLFG}).
All the algorithms were implemented using C++11 with Eigen 3.4.0\footnote{\url{https://eigen.tuxfamily.org/} (accessed Aug 21, 2024)} for matrix computations and compiled by GCC 10.2.0.
Experiments were conducted on a computer with Intel\textregistered\ Xeon\textregistered\ Gold 5222 CPU (\SI{3.8}{GHz}, 4~Cores) and \SI{755}{GB} RAM.
Only double-precision arithmetic has been employed.

\paragraph{Data.}

\begin{table}[tb]
  \caption{Summary of dataset}\label{tbl:dataset}
  \centering
  \begin{tabular}{lrrrrr}\toprule
    \multicolumn{1}{c}{Data} & \multicolumn{1}{c}{$n$} & \multicolumn{1}{c}{$\nnz(K)$} & \multicolumn{1}{c}{$\displaystyle \max_{i \ne j:\: K_{ij} \ne 0} \abs{K_{ij}}$} & \multicolumn{1}{c}{$\displaystyle \min_{i \ne j:\: K_{ij} \ne 0} \abs{K_{ij}}$} \\\midrule
    DFG2 & 1,765 & 7,980 & \num{3.8e+14} & \num{9.4e-229}\\
    WL2  & 1,776 & 8,742 & \num{1.2e+20} & \num{1.7e-149} \\
    ALD2 & 6,199 & 20,848 & \num{3.1e+15} & \num{2.9e-182} \\
    Strecker  & 9,203  & 27,092 & \num{4.2e+17} & \num{1.8e-134}\\
    HDF2      & 10,810 & 29,802 & \num{9.7e+27} & \num{3.1e-210}\\
    Passerini & 12,215 & 30,746 & \num{1.2e+126}& \num{1.2e-144}\\
    \bottomrule
  \end{tabular}
\end{table}

We employ six publicly available chemical reaction path networks from Searching Chemical Action and Network (SCAN) platform~\citep{SCAN}: DFG2 (fluoroglycine synthesis), WL2 (Wohler’s urea synthesis), ALD2 (base-catalyzed aldol reaction), Strecker (strecker reaction), HDF2 (cobalt-catalyzed hydroformylation), and Passerini (Passerini reaction).
Detailed steps for converting a chemical reaction path network into a rate constant matrix can be found, for instance, in~\citep{Sumiya2017-qo}.
Characteristic values of the rate constant matrices are summarized in \cref{tbl:dataset}.

\paragraph{Results.}

\begin{figure}[!p]
  \centering
  \includegraphics[width=\linewidth]{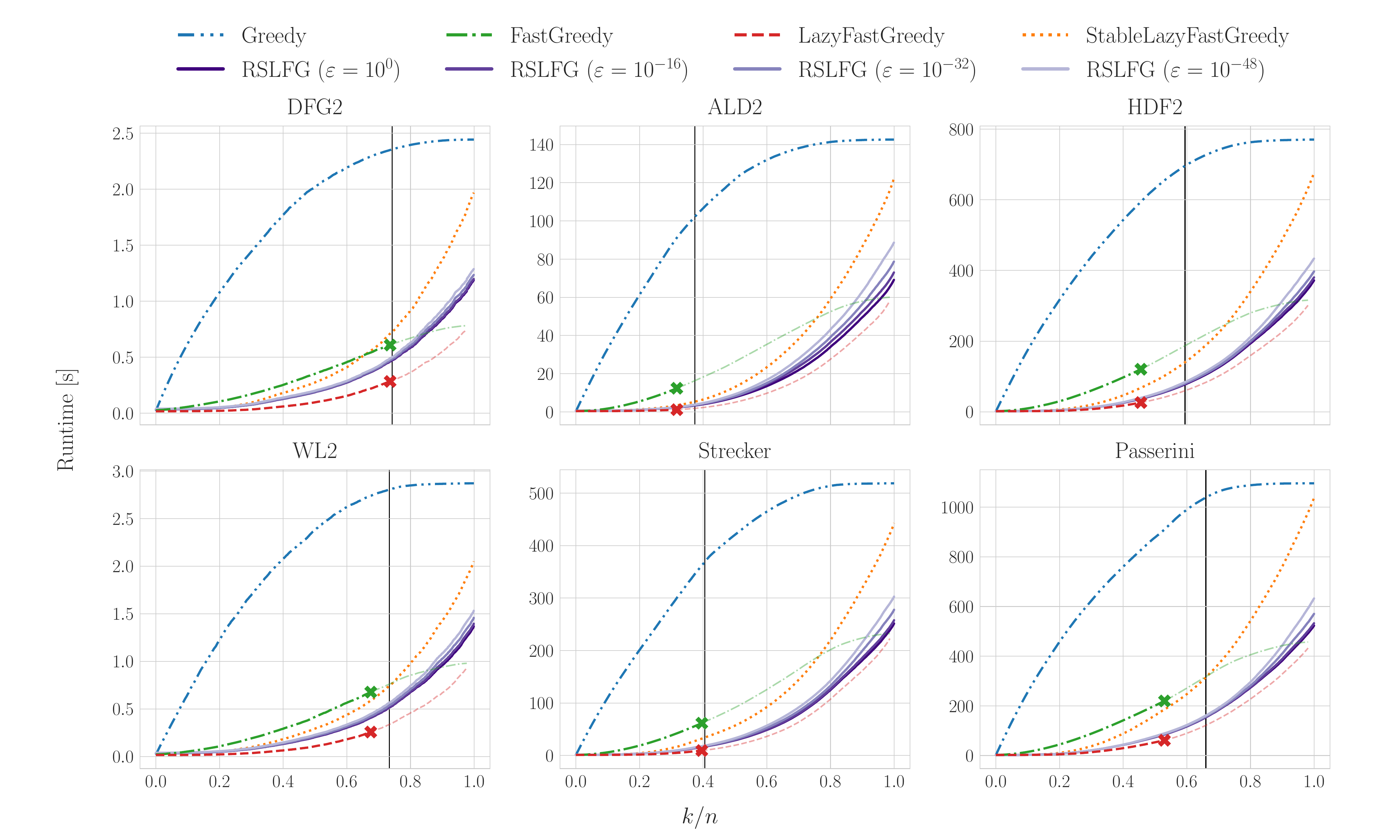}
  \caption{%
    The iteration counts $k$ over the matrices size $n$ versus the running times of \cref{step:1}.
    The cross marks with the horizontal coordinate $k$ indicate the smallest $k$ where the algorithm selected an incorrect element as $s^{(k)}$ due to catastrophic cancellation.
    The vertical black lines indicate the value of $k$ corresponding to $\tmax = \SI{86400}{sec}$.
  }\label{fig:runtime}
  \vspace*{\floatsep}
  \includegraphics[width=\linewidth]{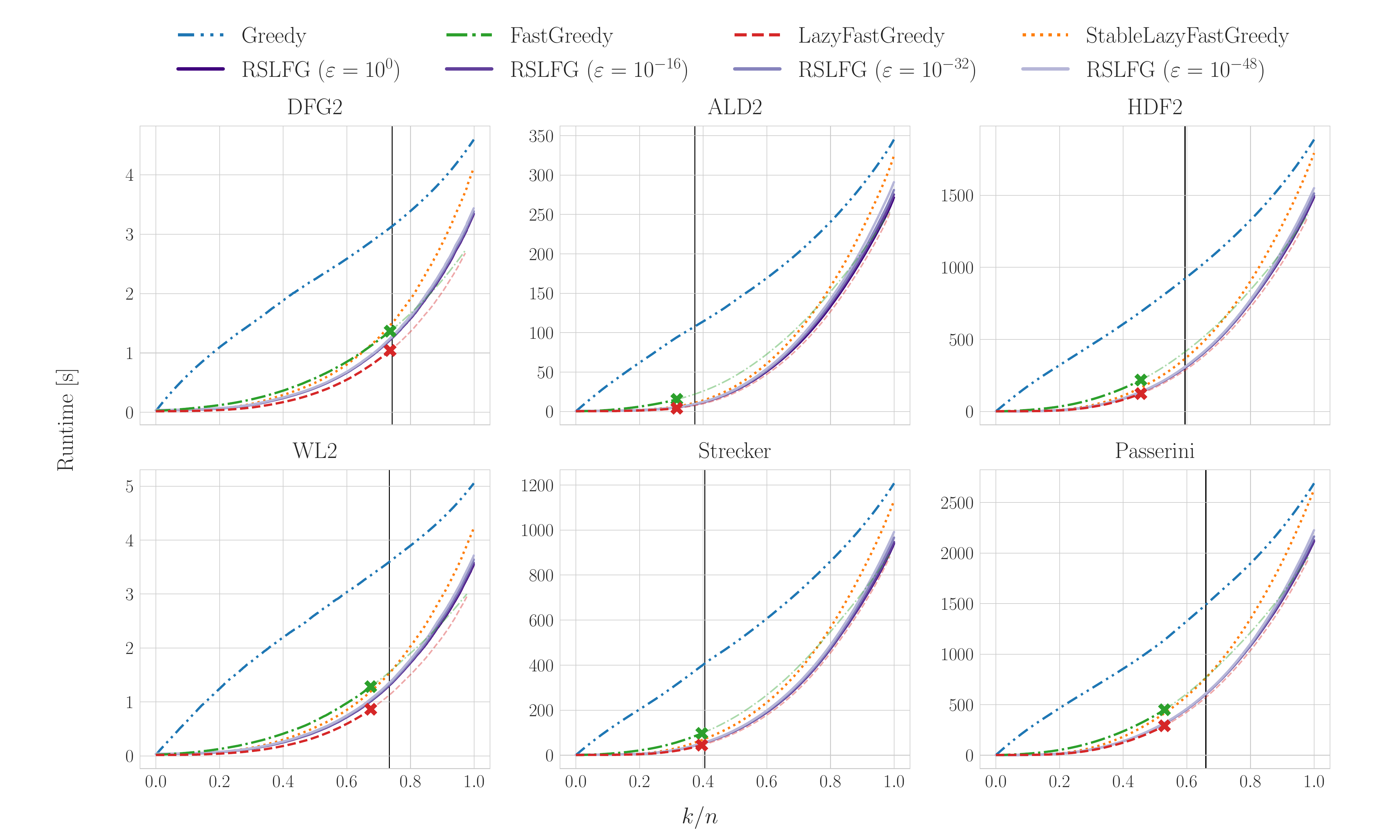}
  \caption{%
    The iteration counts $k$ over the matrices size $n$ versus the total running times of the methods (\cref{step:1,step:2}).
    The cross marks and vertical lines indicate the same as \cref{fig:runtime}.
  }\label{fig:runtime-pop}
\end{figure}

\begin{figure}[!p]
    \centering
    \includegraphics[width=\linewidth]{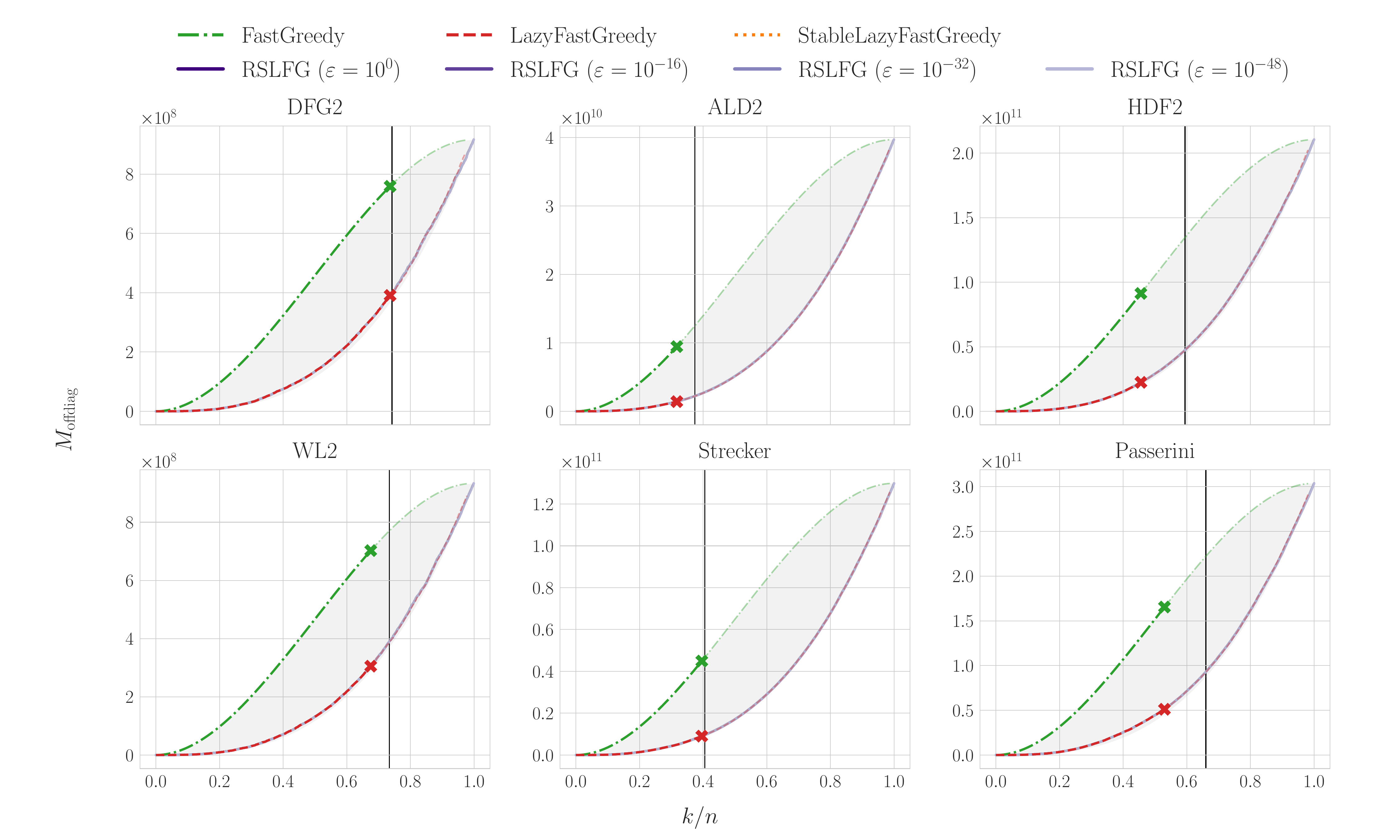}
    \label{subfig:num_computed_offdiagonals_V}
  \caption{
    The iteration counts $k$ over the matrices size $n$ versus $\moffdiag$ (defined in \cref{sec:lazy-greedy-runtime}). The gray areas indicate the ranges of possible values of $\moffdiag$ for every $k$ given by~\eqref{eq:mlazy-bound}.
    The cross marks and vertical lines indicate the same as \cref{fig:runtime}.
  }\label{fig:off-diagonals}
  \vspace*{\floatsep}
  \includegraphics[width=\linewidth]{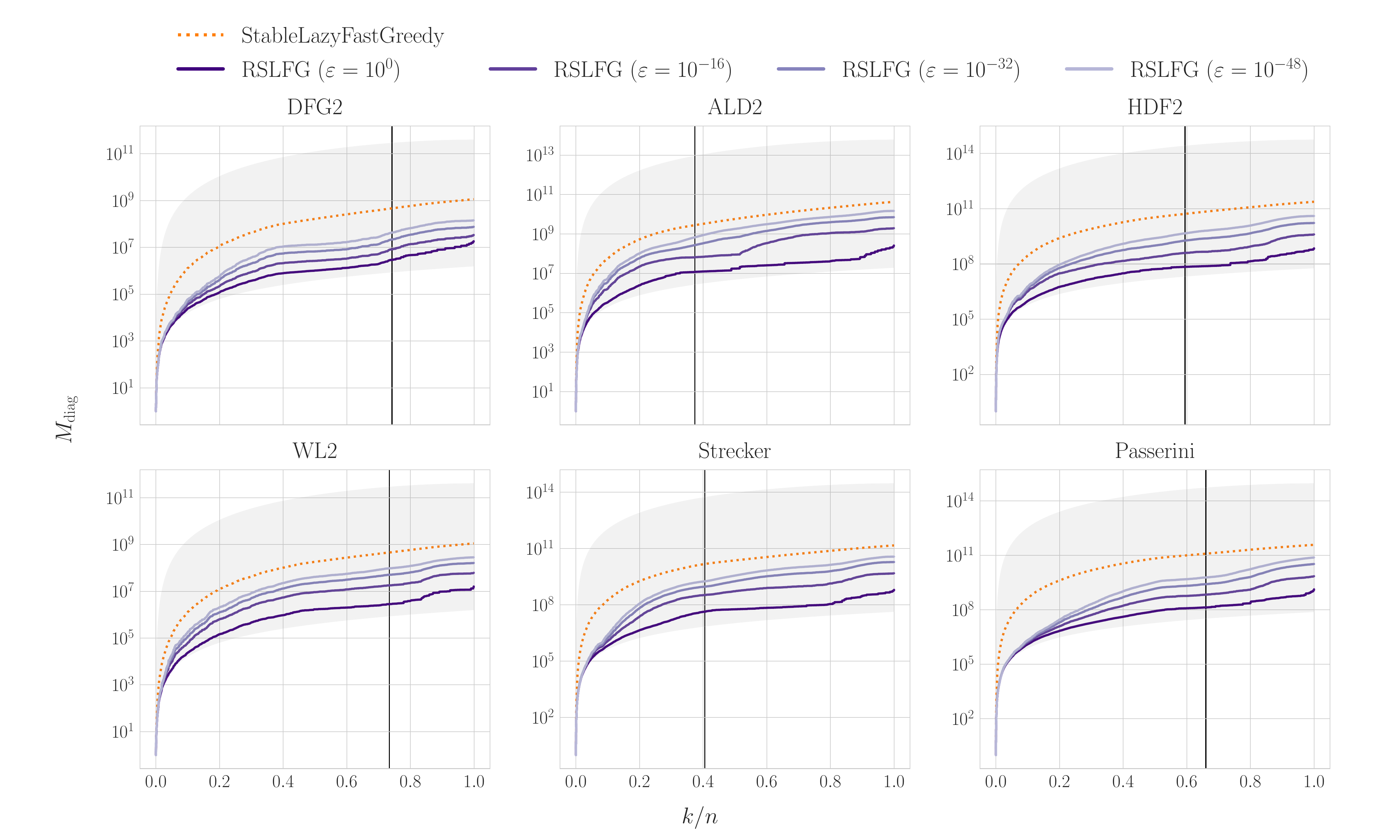}
  \caption{
    The iteration counts $k$ over the matrices size $n$ versus $\mdiag$ (defined in \cref{sec:stable-lazy-runtime}). The gray areas indicate the ranges of possible values of $\moffdiag$ for every $k$ given by \cref{lem:mrelax-bound}.
    The cross marks and vertical lines indicate the same as \cref{fig:runtime}.
  }\label{fig:virtual-off-diagonals}
\end{figure}

\begin{table}[tb]
  \caption{%
    Running times (seconds) at $k$ corresponding to $\tmax = \SI{86400}{sec}$.
    (\textsc{R})\textsc{SLFG} stands for (\textsc{Relaxed})\textsc{StableLazyFastGreedy}.
    The error tolerance $\varepsilon$ in \textsc{RSLFG} is set to $10^{-16}$.
  }\label{tbl:runtime}
  \centering
  \begin{tabular}{lrrrrr}\toprule
    \multicolumn{1}{c}{\multirow{2}[3]{*}{Data}} & \multicolumn{1}{c}{\multirow{2}[3]{*}{$k$}} & \multicolumn{3}{c}{\cref{step:1}} & \multicolumn{1}{c}{\multirow{2}[3]{*}{\cref{step:2}}} \\\cmidrule(lr){3-5}
    & & \multicolumn{1}{c}{\textsc{Greedy}} & \multicolumn{1}{c}{\textsc{SLFG}} & \multicolumn{1}{c}{\textsc{RSLFG}} &  \\\midrule
    DFG2      & \num{1311} & 2.35     & 0.73   & 0.48   & 0.78   \\
    WL2       & \num{1304} & 2.81     & 0.75   & 0.52   & 0.79   \\
    ALD2      & \num{2318} & 102.18  & 5.33   & 3.53   & 5.79    \\
    Strecker  & \num{3729} & 369.14  & 34.16  & 17.33  & 38.39  \\
    HDF2      & \num{6427} & 696.10  & 140.51 & 79.54  & 226.71 \\
    Passerini & \num{8057} & 1038.79 & 314.30 & 154.81 & 450.37 \\
    \bottomrule
  \end{tabular}
\end{table}

\Cref{fig:runtime} shows the running times of the greedy parts (\cref{step:1}) of the algorithms, 
where \textsc{RelaxedStableLazyFastGreedy} was executed with four different error tolerance parameters $\varepsilon \in \set{10^{0}, 10^{-16}, 10^{-32}, 10^{-48}}$.
While \textsc{LazyFastGreedy} ran fastest on all data, it selected incorrect $s^{(j)}$ at some $j$ due to catastrophic cancellations.
\textsc{RelaxedStableLazyFastGreedy} with $\varepsilon = 10^0$ was the fastest among the algorithms that produced the correct solution, whereas that with other values of $\varepsilon$ ran in comparable times.
This means that the choice of $\varepsilon$ does not affect both the numerical stability and the running time significantly.
Considering the fact that even additions of like-sign numbers may increase the error by the factor of the machine epsilon, which is roughly $10^{-16}$ for double-precision numbers, we recommend employing $\varepsilon = 10^{-16}$. 

\Cref{fig:runtime-pop} shows the running times of the methods including both \cref{step:1,step:2} with opt = \texttt{full}.
To measure the running time, we recorded the time $\tau_j$ taken to compute $q^{(j)}$ by \eqref{eq:t-q} for each $j$ and regarded $\sum_{j=1}^k \tau_j$ as the running time of \cref{step:2} until the $k$th iteration.
We observe that \cref{step:1} is actually a bottleneck of the original RCMC method, especially for small $k$, and our fast greedy methods can accelerate the entire procedure even with the Full option.
Note that the running time with opt = \texttt{last} is almost the same as drawn in \cref{fig:runtime} as $\tau_k$ is negligibly small compared to the running time of \cref{step:1}.

\Cref{tbl:runtime} describes the running times of \cref{step:1} (\textsc{Greedy}, \textsc{StableLazyFastGreedy}, and \textsc{RelaxedStableLazyFastGreedy}) and \cref{step:2} at $k$ corresponding to $\tmax = \SI{86400}{sec}$.
We can observe for any data that the running time of \textsc{Greedy} is greater than that of \cref{step:2}, whereas \textsc{RelaxedStableLazyFastGreedy} runs faster than \cref{step:2}.
This means that our proposed method sufficiently speeds up \cref{step:1}, and the bottleneck of the RCMC method is now moved to \cref{step:2}.

To take a closer look at how well the heuristics work, we illustrate in \cref{fig:off-diagonals,fig:virtual-off-diagonals} the values of $\moffdiag$ and $\mdiag$, respectively.
\Cref{fig:off-diagonals} indicate that $\moffdiag$ attains almost the lower bounds in~\eqref{eq:mlazy-bound} at any $k$ for all data, meaning that the lazy heuristics work quite well in practice.
Similarly, the curves of \textsc{RelaxedStableLazyFastGreedy} lie lower than that of \textsc{StableLazyFastGreedy} in \cref{fig:virtual-off-diagonals}, hence the relaxing heuristics is also effective.

\section{Conclusion}\label{sec:conclusion}
In this paper, we have proposed a fast and numerically stable implementation of the greedy part (\cref{step:1}) of the RCMC method for chemical kinetics simulation.
Our technique is to modify the lazy greedy for the DPP MAP inference to avoid subtractions of like-sign numbers, leveraging properties of rate constant matrices.
For faster computation, we employ segment trees and partially allow like-sign subtractions if they do not cause catastrophic cancellations.
Using real instances from chemical reactions, we have confirmed that the presented algorithm runs numerically stably and much faster than the original method.
Accelerating \cref{step:2}, a new bottleneck, is left for future investigation.

\section*{Acknowledgments}
The authors thank Satoshi Maeda and Yu Harabuchi for their helpful comments and Yutaro Yamaguchi for providing information on the reference~\citep{Chazelle1988-zw} of segment trees.
This work was supported by JST ERATO Grant Number JPMJER1903 and JSPS KAKENHI Grant Number JP22K17853.

\printbibliography[heading=bibintoc]

\end{document}